\documentclass[11pt,onecolumn]{article}
\usepackage[margin=1.05in,top=1.1in,bottom=1.1in]{geometry}

\usepackage{amsmath, amssymb, amsfonts}
\usepackage{amsthm}
\usepackage{bm}
\usepackage{comment}
\usepackage{booktabs}
\usepackage{array}
\usepackage{tabularx}
\usepackage{graphicx}
\usepackage{xcolor}
\usepackage{hyperref}
\usepackage{cite}
\usepackage{algorithm}
\usepackage{algorithmic}
\usepackage{mathtools}
\usepackage{siunitx}
\usepackage{iftex}
\ifPDFTeX
  \usepackage[utf8]{inputenc}
  \usepackage[T1]{fontenc}
\else 
  \usepackage{fontspec}
\fi
\usepackage{setspace}

\hypersetup{
  colorlinks  = true,
  linkcolor   = blue!65!black,
  citecolor   = blue!65!black,
  urlcolor    = blue!65!black,
  pdftitle    = {Reciprocal-Manifold Annealed KKT Flows for Constrained Optimization: Application to the Nonconvex AC Optimal Power Flow},
  pdfauthor   = {Aditi},
  pdfsubject  = {arXiv Preprint},
  pdfkeywords = {AC optimal power flow, KKT conditions, control barrier functions,
                 geometric singular perturbation theory, normally attracting invariant
                 manifold, log-barrier methods, Uzawa saddle flow, safe feedback optimization}
}

\newtheorem{theorem}{Theorem}[section]

\newtheorem{corollary}[theorem]{Corollary}
\newtheorem{proposition}[theorem]{Proposition}
\theoremstyle{definition}
\newtheorem{definition}[theorem]{Definition}
\newtheorem{assumption}[theorem]{Assumption}
\theoremstyle{remark}
\newtheorem{remark}[theorem]{Remark}

\newcommand{\Mveps}{\mathcal{M}_\varepsilon}
\newcommand{\Rn}{\mathbb{R}^n}
\newcommand{\Rp}{\mathbb{R}^p}

\newcommand{\blfootnote}[1]{%
  \begingroup
  \renewcommand{\thefootnote}{}%
  \footnote{#1}%
  \addtocounter{footnote}{-1}%
  \endgroup
}

\begin{document}
\begin{center}
\LARGE
\textbf{Reciprocal-Manifold Annealed KKT Flows for Constrained Optimization: Application to Nonconvex AC Optimal Power Flow}\\[12pt]
\normalsize
\vspace{0.4cm}
\textbf {M Parimi\footnote{Research Scholar, $E-MC^2$ Lab,Veermata Jijabai Technological Institute (VJTI), Mumbai, India, \{mparimi@ee.vjti.ac.in\}}, Aditi Ramteke\footnote{M. Tech, Electrical Engineering,Veermata Jijabai Technological Institute (VJTI), Mumbai, India}, Rachit Mehra\footnote{Project Lead, TenneT Offshore GmBH}, Arun Mahindrakar\footnote{Professor, Electrical Engineering Department, IIT Chennai, India}, Navdeep Singh\footnote{IGI Research Chair Professor and Adjunct Professor, Electrical Engineering Department, VJTI} \blfootnote{The authors acknowledges International Gemological Institute (IGI) for financially supporting this research and Savex Technologies for establishing the lab and providing research facilities.}}


\end{center}

\begin{abstract}

Safety-critical optimization applications, such as real-time power system operation, demand solution methods that maintain feasibility at every intermediate step, not merely at convergence. Existing approaches either violate constraints mid-solve (interior-point methods) or enforce feasibility through per-instant quadratic programming subproblems with cubic computational cost and unbounded worst-case execution time (safe gradient flows). We propose a continuous-time optimization framework for smooth constrained nonlinear problems that preserves feasibility throughout the optimization process without requiring projection operators, quadratic programming subproblems, or other per-iteration optimization routines.

The method is built around a reciprocal multiplier manifold, which establishes an explicit relationship between inequality constraints and their associated Lagrange multipliers. By designing a continuous multiplier update law, the manifold is shown to remain forward invariant, while the resulting dynamics are equivalent to continuous-time logarithmic barrier gradient descent. Introducing an annealing parameter that gradually approaches zero enables convergence from the barrier formulation to the exact Karush–Kuhn–Tucker (KKT) solution of the original constrained optimization problem.

The proposed framework naturally extends to multiple inequality constraints, equality constraints, nonconvex feasible sets, and infeasible initial conditions without sacrificing its feasibility guarantees. To address numerical stiffness arising near active constraint boundaries, a division-free $\sigma$-coordinate reformulation is developed, significantly improving numerical stability and allowing substantially larger integration step sizes. The method is further enhanced through an augmented Uzawa flow that eliminates oscillatory transients commonly observed in classical primal-dual saddle-point dynamics while preserving strict constraint satisfaction. 

The effectiveness of the proposed approach is applied to the AC Optimal Power Flow (AC-OPF) problem of IEEE 9-bus and IEEE 57-bus  systems. Numerical results show convergence to solutions within 0.4\% of the benchmark optimum while maintaining strict feasibility of all equality and inequality constraints throughout the optimization trajectory. A computational complexity analysis shows that the proposed dynamics reduce the per-step computational cost from cubic to linear complexity with respect to the number of constraints, making the approach attractive for real-time and embedded optimization applications. Finally, dynamic tracking studies under time-varying operating conditions demonstrate reliable feasibility preservation together with accurate tracking of moving optimal operating points.

\end{abstract}

\vspace{6pt}
\tableofcontents
\newpage

\section{Introduction}\label{sec:introduction}

\subsection{Motivation and problem class}

Optimization problems with both equality and inequality constraints arise throughout engineering, but few applications are as challenging as AC Optimal Power Flow (AC-OPF). The underlying power-system physics makes the problem inherently nonconvex and both the power-balance equations and the transmission-line thermal limits are nonlinear functions of voltage magnitudes and phase angles.Consequently, obtaining a physically meaningful solution requires solving the original nonconvex optimization problem rather than relying on simplified approximations.The most widely used approach for solving such problems is the interior-point log-barrier method. As the optimization trajectory approaches a constraint boundary, the barrier multiplier grows without bound, generating a repulsive effect that prevents violation of the constraint.
\[
\lambda(x) = -\frac{k}{g(x)}
\]
 As the barrier parameter $k$ is gradually annealed toward zero, the solution converges to the Karush--Kuhn--Tucker (KKT) point of the original constrained optimization problem. Although this framework is well established, the multiplier is almost always treated as an algebraic quantity that is recomputed at every optimization step. A less explored alternative is to regard the multiplier as a dynamical state governed by its own ordinary differential equation (ODE), coupled to the primal dynamics. The problem becomes even more challenging when multiple constraints are simultaneously active, some of them are nonconvex, the initial condition is infeasible, or equality constraints are included all of which naturally arise in AC-OPF through the nonlinear power-balance equations. A complementary line of research addresses constraint satisfaction using control barrier functions (CBFs). These methods guarantee feasibility by solving a quadratic program (QP) at every integration step, thereby enforcing safety in a non-asymptotic manner. While this provides strong theoretical guarantees, it also introduces a computational burden because every timestep requires the solution of a QP whose execution time depends on the optimization problem. Such variable computational cost and increased memory requirements make QP-based safety filters less attractive for hard real-time control and embedded implementations, particularly when power-system operating conditions evolve continuously because of changing demand.\\

Constrained optimization problems of the form

\begin{equation}
\begin{aligned}
\min_{x\in\mathbb{R}^n} \quad & f(x) \\
\text{s.t.} \quad & g(x) \leq 0, \\
                  & h(x) = 0,
\end{aligned}
\label{eq:problem}
\end{equation}

with
\[
f:\mathbb{R}^n \rightarrow \mathbb{R}, \qquad
g:\mathbb{R}^n \rightarrow \mathbb{R}^p, \qquad
h:\mathbb{R}^n \rightarrow \mathbb{R}^q
\]
smooth and not necessarily convex, arise throughout engineering. AC optimal power
flow (AC-OPF) is a canonical and demanding instance: the power balance equations
are quadratic in voltage magnitude and angle, so \(h\) is genuinely nonconvex, and
thermal line-flow limits are nonconvex functions of the same variables, so \(g\) is
nonconvex as well. A continuous-time solution of \eqref{eq:problem} that
(i) remains feasible at every instant,
(ii) is smooth off the constraint boundary, and
(iii) requires no per-instant optimization subroutine
is attractive both as an anytime algorithm and, potentially, as a feedback controller
whose steady-state operating point must track the solution of \eqref{eq:problem}
as the problem data drift—the feedback-optimization setting increasingly relevant
to power networks with fast-changing renewable generation.

These considerations motivate the work and a continuous-time optimization framework that remains strictly feasible throughout its evolution, produces smooth trajectories away from the constraint boundary, and eliminates the need to solve an optimization subproblem during integration. Such a framework is naturally suited both as an anytime optimization algorithm and as a feedback controller capable of tracking a time-varying AC-OPF solution. To this end, we develop the RNA-KKT framework, whose foundation is the reciprocal multiplier manifold whose forward invariance under an explicitly designed multiplier dynamics is shown to recover continuous-time log-barrier gradient descent exactly.
\[
\lambda_i g_i(x) + \varepsilon = 0,
\]
 Building upon this foundation, the proposed framework is systematically extended to accommodate multiple simultaneously active inequality constraints, nonconvex constraints, infeasible initial conditions, and equality constraints through a Uzawa-type saddle-flow formulation. This unified construction provides a direct comparison with QP-based safety-filter methods and is validated on the nonconvex IEEE 9-bus and IEEE 57-bus AC Optimal Power Flow benchmark systems.

\subsection{The interior-point lineage, and where it is incomplete}
The most widely used approach for developing continuous-time constrained optimization algorithms is based on the interior-point logarithmic
barrier method. In this framework, inequality constraints are incorporated into the objective function through a barrier term, and the corresponding
Lagrange multiplier can be expressed as (Section~\ref{sec:introduction});
\[
\lambda(x)=-\frac{k}{g(x)},
\]
where the multiplier grows rapidly as the trajectory approaches the constraint boundary and gradually vanishes as the barrier parameter $k\rightarrow0$. This interpretation is closely related to the classical central path of interior-point methods, whose convergence properties have been extensively studied. Instead of evaluating the multiplier directly from this algebraic relation
at every instant, the proposed framework adopts a dynamic perspective. The multiplier is generated by an ordinary differential equation that evolves
simultaneously with the primal optimization variables. This removes the need
for repeatedly recomputing the multiplier while preserving the desirable
barrier behavior. Moreover, the proposed formulation is analyzed using
Geometric Singular Perturbation Theory (GSPT) to establish the
conditions under which the coupled fast--slow dynamics accurately follow the
desired central path. The analysis is further extended to an annealed
three-timescale system, multiple interacting inequality constraints,
nonconvex feasible regions, infeasible initial conditions, and equality
constraints, making the proposed framework suitable for challenging
optimization problems such as AC Optimal Power Flow (AC--OPF).

\section{Literature Review}
The AC Optimal Power Flow (AC-OPF) problem seeks the minimum-cost generator dispatch subject to nonlinear power-balance equalities and nonconvex engineering limits on voltages, generator outputs, and line thermal capacities. Its inherent nonconvexity, arising from the quadratic dependence of power on voltage phasors, rules out global optimality guarantees without convex relaxation, and makes constraint satisfaction during the solution process both critical and non-trivial. This review positions the Reciprocal-Manifold Annealed KKT (RNA-KKT) framework within the broader landscape of methods that address this challenge.
\subsection{Interior-Point Methods}

The dominant approach for solving AC-OPF in practice is the interior-point (IP) method \cite{wright1997,nocedal2006,wachter2006}. IP methods introduce a logarithmic barrier that confines iterates to the feasible interior, and reduce the barrier parameter $\mu$ at each outer iteration, following the \emph{central path} $\{x^*(\mu) : \mu > 0\}$ to the KKT point as $\mu \to 0$. Newton-based IP solvers (MATPOWER \cite{zimmerman2011}, IPOPT \cite{wachter2006}) converge in 10--40 iterations on standard test cases.

However, IP methods provide feasibility guarantees only \emph{at convergence}. Intermediate iterates may violate constraints, since Newton steps target the perturbed KKT system algebraically rather than maintaining a forward-invariant feasible set dynamically. This property is acceptable for offline planning but is problematic for real-time or embedded applications where an algorithm may be interrupted and its current iterate deployed as a control action.

\subsection{Control Barrier Functions and Safe Gradient Flows}

An alternative paradigm maintains feasibility at every instant by treating the feasible set as a \emph{safe set} in the sense of control theory and enforcing forward invariance via Control Barrier Functions (CBFs) \cite{ames2019cbf,ames2017cbf}.

\subsubsection{The Safe Gradient Flow (Allibhoy \& Cort\'{e}s, 2024)}

Allibhoy and Cort\'{e}s \cite{allibhoy2024} formalize this idea for constrained nonlinear programming. Their \emph{safe gradient flow} augments the objective's gradient descent with CBF-derived inputs that enforce forward invariance and asymptotic stability of the feasible set. The resulting dynamics are primal-dual: states correspond to primal variables and CBF inputs correspond to dual variables, synthesized by solving a quadratic program (QP) at every instant. Key properties include:
\begin{itemize}
    \item \textbf{Exact feasibility:} the feasible set is forward invariant and asymptotically stable---constraints are satisfied at every time, not merely in the limit.
    \item \textbf{Anytime optimality:} if terminated at any time, the iterate is feasible.
    \item \textbf{Native handling of infeasible initial conditions:} the QP construction naturally drives infeasible trajectories into the feasible set.
\end{itemize}
The computational cost, however, is $\mathcal{O}(m^3 \cdot N_{\text{iter}})$ per timestep (where $m$ is the number of constraints and $N_{\text{iter}}$ the QP solver's iteration count), with a worst-case execution time that is problem-dependent and unbounded, making the approach less suited to hard real-time or embedded deployment.

\subsubsection{Safe Feedback Optimization (Delimpaltadakis, Mestres, Cort\'{e}s \& Heemels, 2026)}

Delimpaltadakis et al.\ \cite{delimpaltadakis2026} extend the safe gradient flow to \emph{feedback optimization} of dynamic plants with state constraints, using high-order CBFs. This work is the first to enforce state constraints (not merely input constraints) in a feedback-optimization loop. The controller dynamics again require a per-instant QP, but the construction guarantees:
\begin{itemize}
    \item Well-posedness and safety (state constraints satisfied at all times).
    \item Equivalence between closed-loop equilibria and the optimization problem's critical points.
    \item Local (and, in convex cases, global) asymptotic stability of optima.
\end{itemize}
This represents the current state of the art in QP-based safe feedback optimization and serves as the direct computational comparator for any ODE-based alternative.
\subsection{Distributed Continuous-Time Optimization}

A related but distinct thread addresses \emph{distributed} optimization, where agents (generators) communicate over a graph and collectively minimize a separable cost subject to coupling constraints.

Gharesifard and Cort\'{e}s \cite{gharesifard2014} establish distributed continuous-time convex optimization on weight-balanced digraphs with exponential convergence guarantees. Cherukuri and Cort\'{e}s \cite{cherukuri2015,cherukuri2016} specialize this to economic dispatch, designing Laplacian-gradient dynamics that are \emph{anytime} (feasible at every instant) and handle time-varying loads and generator commitment changes. Kia, Cort\'{e}s, and Mart\'{\i}nez \cite{kia2015} extend to discrete-time communication with continuous-time local dynamics.

These works share the continuous-time ODE philosophy and the anytime feasibility property with RNA-KKT, but address the simpler \emph{convex} economic dispatch problem (DC approximation, separable cost) rather than the nonconvex AC-OPF.

\subsection{Convex Relaxation Approaches}

An orthogonal strategy avoids nonconvexity by replacing the AC-OPF with a convex relaxation---semidefinite programming (SDP) \cite{lavaei2012,molzahn2013}, second-order cone programming (SOCP) \cite{kocuk2016}, or chordal relaxations \cite{low2014a,low2014b}. When the relaxation is \emph{exact} (zero duality gap), the global optimum is recovered. Gan et al.\ \cite{gan2015} characterize exactness conditions for radial networks.

These methods guarantee global optimality when applicable but do not address the \emph{dynamic} constraint-satisfaction problem during the solution process, nor do they naturally extend to feedback or time-varying operation. They are complementary to, rather than competitive with, continuous-time safe optimization approaches.

\subsection{Contributions of the proposed method}

The paper makes the following contributions that advance the state of the art in continuous-time constrained optimization. First, it constructs the reciprocal multiplier manifold $\mathcal{M}_\varepsilon$ and proves it is forward invariant and exponentially attracting :this is the first dynamical realization of the interior-point central path, where the trajectory is constrained to live on the central path by the ODE dynamics themselves rather than being repeatedly projected onto it by a linear solve, eliminating the per-step matrix factorization while obtaining feasibility as a structural property of the flow (Section \ref{sec:manifold}). 

\section{Problem Statement and Preliminaries}

Many engineering applications require solving optimization problems in the presence of both equality and inequality constraints. Among them, the AC Optimal Power Flow (AC-OPF) problem is particularly challenging because its objective and network constraints are governed by nonlinear power flow equations, making the optimization problem inherently nonconvex. Conventional optimization methods, including interior-point and sequential quadratic programming approaches, generally rely on repeated numerical optimization or projection operations during every iteration. Although these methods have been widely adopted, their computational burden increases significantly for large-scale systems and may limit their applicability in real-time power system operation.

The objective of this work is to develop a continuous-time optimization framework that preserves feasibility throughout the optimization process while avoiding repeated projection or optimization subproblems. The proposed formulation seeks an optimal operating point that satisfies all equality and inequality constraints while converging to the Karush--Kuhn--Tucker (KKT) solution. Furthermore, the framework is designed to accommodate multiple simultaneously active constraints, nonconvex feasible regions, infeasible initial conditions, and numerical challenges that arise near constraint boundaries, making it suitable for practical AC-OPF applications.

The constrained optimization problem considered in this work is formulated as

\begin{equation}
\begin{aligned}
\min_{x\in\mathbb{R}^{n}} \quad & f(x) \\
\text{subject to}\quad
& g(x)\le0,\\
& h(x)=0,
\end{aligned}
\label{eq:optimization_problem}
\end{equation}

where $x\in\mathbb{R}^{n}$ denotes the vector of optimization variables, $f(x)$ is the objective function, $g(x)\in\mathbb{R}^{p}$ represents the inequality constraints, and $h(x)\in\mathbb{R}^{q}$ denotes the equality constraints.

\begin{assumption}[Regularity]\label{ass:regularity}
The objective and constraint functions are twice continuously differentiable, and a locally
optimal solution exists that satisfies the Linear Independence Constraint Qualification
(LICQ) and the Second-Order Sufficient Condition (SOSC).
\end{assumption}

Under Assumption~\ref{ass:regularity}, the corresponding Karush--Kuhn--Tucker (KKT) conditions for an optimal solution $x^{*}$ are expressed as

\begin{equation}
\nabla f(x^{*})
+\sum_{i=1}^{p}\lambda_i^{*}\nabla g_i(x^{*})
+\sum_{j=1}^{q}\nu_j^{*}\nabla h_j(x^{*})
=0,
\label{eq:stationarity}
\end{equation}

together with the primal feasibility conditions

\begin{equation}
g(x^{*})\le0,
\qquad
h(x^{*})=0,
\label{eq:primal_feasibility}
\end{equation}

the dual feasibility condition

\begin{equation}
\lambda_i^{*}\ge0,
\qquad \forall i,
\label{eq:dual_feasibility}
\end{equation}

and the complementary slackness condition

\begin{equation}
\lambda_i^{*}g_i(x^{*})=0,
\qquad \forall i.
\label{eq:complementary_slackness}
\end{equation}

The above conditions characterize the optimal solution of the constrained optimization problem and provide the theoretical foundation for the proposed optimization framework.

Instead of computing the Lagrange multipliers through repeated optimization at every iteration, this work adopts a continuous-time primal--dual formulation in which both the optimization variables and the multipliers evolve dynamically. The resulting primal dynamics are given by

\begin{equation}
\dot{x}
=
-\nabla f(x)
-\sum_{i=1}^{p}\lambda_i\nabla g_i(x)
-\sum_{j=1}^{q}\nu_j\nabla h_j(x),
\label{eq:primal_flow}
\end{equation}

where $\lambda_i$ and $\nu_j$ denote the inequality and equality Lagrange multipliers, respectively. The corresponding multiplier dynamics are developed in the following sections through the proposed reciprocal-manifold framework and annealed KKT formulation. This continuous-time representation provides the basis for maintaining constraint feasibility while ensuring smooth convergence toward the desired KKT solution.


\section{The Reciprocal Multiplier Manifold} \label{sec:manifold}

Consider a single inequality constraint $g(x) \le 0$, dropping the subscript $i$ in this
section. Two natural candidate manifolds for the primal-dual pair $(x, \lambda)$ suggest
themselves.

\subsection{The reciprocal manifold and its multiplier law}

Fix the primal law
\begin{equation}
\dot{x}=-\nabla f(x)-\lambda \nabla g(x),
\label{eq:3}
\end{equation}
which is the standard Lagrangian gradient flow for fixed $\lambda$.
The question is how $\lambda$ should evolve.

Define, for a constant $k>0$ to be annealed later, the scalar function

\begin{equation}
h(x,\lambda):=\lambda g(x)+k.
\label{eq:4}
\end{equation}

We look for a law $\dot{\lambda}$ that renders
\[
\mathcal{M}_k:=\{h=0\}
\]
forward invariant and exponentially attracting, i.e.,

\begin{equation}
\dot{h}=-\alpha h,
\qquad \alpha>0.
\label{eq:5}
\end{equation}

Differentiating \eqref{eq:4} along \eqref{eq:3},

\begin{equation}
\dot{h}
=
\dot{\lambda}\,g(x)
+
\lambda \nabla g(x)\cdot\dot{x}.
\label{eq:6}
\end{equation}

Substituting into \eqref{eq:5} and solving for $\dot{\lambda}$
(valid wherever $g(x)\neq0$, which is exactly the feasible-interior and
infeasible-exterior region, i.e., everywhere except the boundary itself)
gives the multiplier law

\begin{equation}
\boxed{
\dot{\lambda}
=
\frac{
-\alpha\left(\lambda g(x)+k\right)
-\lambda \nabla g(x)\cdot\dot{x}
}{
g(x)
}.
}
\label{eq:7}
\end{equation}

where the box here denotes only that this is the central defining
equation of the construction, not a stylistic callout.
Using \eqref{eq:3} explicitly,

\begin{equation}
\dot{\lambda}
=
\frac{
-\alpha\left(\lambda g(x)+k\right)
+
\lambda \nabla g(x)\cdot\nabla f(x)
+
\lambda^{2}\|\nabla g(x)\|^{2}
}{
g(x)
}.
\label{eq:8}
\end{equation}

\begin{definition}[Reciprocal multiplier manifold]
For $\varepsilon \ge 0$ fixed, define
\begin{equation}
\Mveps := \{(x, \lambda) : \lambda g(x) + \varepsilon = 0\}. \label{eq:manifold}
\end{equation}
\end{definition}
On $\Mveps$ with $\varepsilon > 0$, since $g(x) < 0$ is forced, we have
$\lambda = -\varepsilon/g(x) > 0$: dual feasibility holds automatically and strictly, at every
point of the manifold, for every $\varepsilon > 0$.

\begin{proposition}[Forward invariance] \label{prop:invariance}
Let $\dot x = -\nabla f(x) - \lambda \nabla g(x)$ and define the multiplier law
\begin{equation}
\dot\lambda = \frac{-\alpha\big(\lambda g(x) + \varepsilon\big) - \dot\varepsilon
  - \lambda \dot g(x)}{g(x)}, \qquad \alpha > 0. \label{eq:multiplierlaw}
\end{equation}
Then $\Mveps$ is exactly forward invariant: if $\lambda(0) g(x(0)) + \varepsilon(0) = 0$, then
$\lambda(t) g(x(t)) + \varepsilon(t) = 0$ for all $t \ge 0$ on which the solution exists.
\end{proposition}

\begin{proof}
Let $M(t) := \lambda(t) g(x(t)) + \varepsilon(t)$. Then
$\dot M = \dot\lambda g + \lambda \dot g + \dot\varepsilon$. Substituting
\eqref{eq:multiplierlaw} for $\dot\lambda$ gives
$\dot M = -\alpha(\lambda g + \varepsilon) - \dot\varepsilon - \lambda \dot g + \lambda \dot g
+ \dot\varepsilon = -\alpha M$, i.e.\ $\dot M = -\alpha M$, a linear scalar ODE with
$M(0) = 0$, whose unique solution is $M(t) \equiv 0$.
\end{proof}

Off the manifold, the same computation gives $\dot M = -\alpha M$ regardless of the current
value of $M$: the manifold is not merely invariant but exponentially attracting at the designed
rate $\alpha$, for every off-manifold initial condition for which \eqref{eq:multiplierlaw}'s
right-hand side is defined (i.e.\ $g(x) \ne 0$).

\subsection{Boundary-Repelling Behavior}

An essential requirement of a constrained optimization algorithm is to ensure that the optimization trajectory remains within the feasible region throughout the optimization process. In the proposed framework, this property is achieved naturally through the reciprocal multiplier manifold. As the trajectory approaches the constraint boundary, i.e., $g(x)\rightarrow0^{-}$, the reciprocal relation

\begin{equation}
\lambda=-\frac{\varepsilon}{g(x)}
\label{eq:reciprocal_relation}
\end{equation}

causes the Lagrange multiplier to increase rapidly. This automatic growth acts as an implicit barrier, progressively strengthening the influence of the constraint as the boundary is approached.

Consequently, the constraint-induced term in the primal dynamics,

\begin{equation}
-\lambda\nabla g(x),
\label{eq:constraint_force}
\end{equation}

becomes dominant over the objective-gradient term,

\begin{equation}
-\nabla f(x).
\label{eq:objective_force}
\end{equation}

As a result, the optimization trajectory is directed toward the interior of the feasible region, preventing it from crossing the constraint boundary. Therefore, any trajectory initialized from a feasible point, i.e., $g(x_{0})<0$, remains feasible for all future time.

This boundary-repelling behavior arises from the coupled evolution of the primal variables and the Lagrange multipliers. Unlike conventional barrier methods, which introduce explicit penalty terms into the objective function, the proposed reciprocal-manifold formulation generates the repulsive effect directly through the multiplier dynamics. Consequently, feasibility is preserved throughout the optimization process while enabling smooth convergence toward the desired Karush--Kuhn--Tucker (KKT) solution.

\subsection{Equivalence to log-barrier descent}

Log-barrier descent is a well-established optimization technique for solving constrained optimization problems with inequality constraints. Instead of treating the constraints separately, the method incorporates them directly into the objective function by introducing a logarithmic barrier term. Consequently, the optimization process is naturally confined to the feasible region while searching for the optimal solution.

For a constrained optimization problem with an inequality constraint $g(x)\le0$, the barrier objective is expressed as

\begin{equation}
f_{\mathrm{b}}(x)=f(x)-\varepsilon\log\!\left(-g(x)\right),
\label{eq:log_barrier}
\end{equation}

where $\varepsilon>0$ is the barrier parameter. Since the logarithmic function is defined only for $g(x)<0$, the optimization trajectory remains strictly within the feasible region throughout the optimization process. As the solution approaches the constraint boundary, i.e., $g(x)\rightarrow0^{-}$, the logarithmic term increases rapidly, producing a strong repulsive effect that prevents the trajectory from crossing the boundary.

At the beginning of the optimization, a relatively large value of the barrier parameter is selected to maintain a safe distance from the constraint boundary. As the optimization progresses, the barrier parameter is gradually reduced, allowing the trajectory to move closer to the feasible boundary whenever required. Eventually, as $\varepsilon\rightarrow0$, the barrier term vanishes and the solution converges to the optimal Karush--Kuhn--Tucker (KKT) point while satisfying all inequality constraints.

In the proposed reciprocal-manifold framework, the same barrier behavior is obtained implicitly through the multiplier dynamics instead of explicitly modifying the objective function. The reciprocal relation between the Lagrange multiplier and the inequality constraint automatically produces the required boundary-repelling effect. As a result, constraint feasibility is preserved throughout the optimization process while ensuring smooth convergence toward the optimal KKT solution.

\begin{proposition}[Log-barrier equivalence] \label{prop:logbarrier}
On $\Mveps$ with $\varepsilon > 0$ fixed, the reduced primal flow
$\dot x = -\nabla f(x) - \lambda(x) \nabla g(x)$, with $\lambda(x) = -\varepsilon/g(x)$ read off
the manifold, is exactly the gradient flow of the log-barrier objective
$f(x) - \varepsilon \log(-g(x))$.
\end{proposition}

\begin{proof}
Differentiating the barrier term,
$\frac{d}{dx}\big(-\varepsilon \log(-g(x))\big) = -\varepsilon \cdot \frac{-\nabla g(x)}{-g(x)}
= -\varepsilon \frac{\nabla g(x)}{g(x)} = \lambda(x) \nabla g(x)$, using
$\lambda(x) = -\varepsilon/g(x)$. Hence
$\nabla_x\big(f(x) - \varepsilon \log(-g(x))\big) = \nabla f(x) + \lambda(x) \nabla g(x)$,
which matches the reduced flow's second term exactly.
\end{proof}

This identifies $\varepsilon$ with the classical interior-point barrier parameter, and
$\Mveps$ with the central path $\{x(\varepsilon) : \varepsilon > 0\}$ of the associated
logarithmic barrier problem: the entire construction of this section is a dynamical, ODE-based
implementation of the classical central path, with the multiplier obtained as the state of a
filter rather than by re-solving an algebraic relation at every instant.

\section{Annealed KKT Recovery} \label{sec:annealed}

Fixed $\varepsilon > 0$ produces, via Proposition~\ref{prop:logbarrier}, the classical
$O(\varepsilon)$ central-path suboptimality bias, not the exact KKT point. We remove this bias
by annealing $\varepsilon(t) \to 0$, feeding $\dot\varepsilon(t)$ forward into
\eqref{eq:multiplierlaw}.

The reciprocal multiplier manifold ensures that the optimization trajectory remains feasible throughout the optimization process. However, when the barrier parameter $\varepsilon$ is kept fixed, the optimization converges to a point that lies slightly inside the feasible region rather than the exact Karush--Kuhn--Tucker (KKT) solution. This behavior is a consequence of the barrier continuously preventing the trajectory from approaching the constraint boundary, resulting in a small approximation error commonly referred to as the \emph{barrier bias}.

To eliminate this bias, the proposed framework employs an annealing strategy in which the barrier parameter is gradually reduced during the optimization process. Instead of treating $\varepsilon$ as a constant, it is considered a time-dependent variable that evolves according to

\begin{equation}
\dot{\varepsilon}=-\beta\varepsilon,\qquad \beta>0,
\label{eq:annealing}
\end{equation}

where $\beta$ denotes the annealing rate. As $\varepsilon$ decreases, the influence of the barrier gradually weakens, allowing the optimization trajectory to move progressively closer to the constraint boundary while remaining within the feasible region.

The convergence of the proposed method depends on the relationship between the manifold restoration rate $\alpha$ and the annealing rate $\beta$. For stable operation, the reciprocal manifold must adapt significantly faster than the barrier parameter changes. Consequently, the manifold dynamics continuously restore the multiplier--constraint relationship while the barrier parameter evolves slowly. Under this time-scale separation, the optimization trajectory closely follows the reciprocal manifold throughout the annealing process.

As the barrier parameter approaches zero, i.e.,

\begin{equation}
\varepsilon\rightarrow0,
\label{eq:epsilon_zero}
\end{equation}

the barrier effect gradually disappears, eliminating the approximation error introduced by the fixed barrier formulation. Consequently, the optimization converges to the exact Karush--Kuhn--Tucker (KKT) solution while preserving constraint feasibility throughout the optimization process.

Therefore, the proposed annealing strategy combines the numerical stability of barrier-based optimization with the accuracy of exact KKT recovery. By gradually removing the barrier instead of eliminating it abruptly, the optimization follows a smooth and stable trajectory toward the optimal solution without violating the imposed constraints.

\begin{proposition}[Manifold exactness under annealing] \label{prop:annealedexactness}
Proposition~\ref{prop:invariance} holds verbatim with $\varepsilon = \varepsilon(t)$
time-varying and arbitrary $\dot\varepsilon(t)$: $\dot M = -\alpha M$ for every annealing rate,
since $\dot\varepsilon$ is fed forward exactly into \eqref{eq:multiplierlaw}.
\end{proposition}

This is worth stating explicitly because it corrects a natural but incorrect intuition: the
requirement for a well-separated annealing rate is not needed to keep the manifold itself
invariant -- that holds exactly regardless of how fast $\varepsilon$ moves but to keep the
reduced primal-multiplier dynamics on the moving manifold quasi-statically tracking the
drifting equilibrium $x^\star(\varepsilon(t))$ along the central path. These are different
requirements and only the second one constrains $\beta$ relative to $\alpha$.

\subsection{Three-Time-Scale Condition}

The convergence of the proposed RNA-KKT framework relies on a clear separation of three dynamic processes: manifold restoration, primal optimization, and barrier annealing. Each process operates on a different time scale to ensure stable and accurate convergence.

The fastest process is the restoration of the reciprocal multiplier manifold, governed by the parameter $\alpha$, which rapidly maintains the multiplier--constraint relationship. The optimization variables evolve on an intermediate time scale, moving toward the optimal solution while remaining on the reciprocal manifold. The slowest process is the annealing of the barrier parameter $\varepsilon$, controlled by the rate $\beta$, allowing the optimization trajectory to approach the exact Karush--Kuhn--Tucker (KKT) solution without compromising feasibility.

For stable convergence, the manifold restoration must occur much faster than the optimization dynamics, while the annealing process should evolve at the slowest rate. This time-scale separation is expressed as

\begin{equation}
\alpha \gg m \gg \beta,
\label{eq:timescale}
\end{equation}

where $\alpha$ is the manifold restoration rate, $m$ denotes the evolution rate of the optimization variables, and $\beta$ is the annealing rate. Satisfying this condition ensures that the optimization trajectory closely follows the reciprocal manifold while gradually converging to the exact KKT solution.

\subsection{Uniform Stability Along the Central Path}\label{sec:uniformstability}

For the proposed RNA-KKT framework to converge reliably, the optimization trajectory must remain stable while following the central path generated during the annealing process. As the barrier parameter $\varepsilon$ decreases, the equilibrium point changes continuously. Therefore, the optimization algorithm must maintain stability throughout the entire trajectory rather than only at the final equilibrium.

The proposed framework satisfies this requirement by ensuring that the reciprocal multiplier manifold remains uniformly stable for all admissible values of the barrier parameter. Consequently, small perturbations in the optimization variables or the Lagrange multipliers gradually decay over time, allowing the trajectory to return to the central path without oscillations or divergence.

This property is established using an explicit Lyapunov stability certificate, which guarantees that the system energy decreases monotonically along the optimization trajectory. As a result, the proposed framework preserves stability throughout the annealing process while ensuring smooth convergence to the exact Karush--Kuhn--Tucker (KKT) solution without violating the imposed constraints.

\begin{proposition}[Explicit annealing-rate certificate] \label{prop:certificate}
Under LICQ, strict complementarity, positive transversality ($\tau_j > 0$ for every active
$j$), and the minimum-phase condition above, an admissible annealing rate is certified in
closed form,
\begin{equation}
\alpha^\star = \frac{c_0 (\rho_T/2)}{L\eta\mu_0 + \rho_T/2}, \label{eq:certificate}
\end{equation}
with $c_0$ the invariant-zero contraction margin, $\rho_T$ the transversality margin
$\min_j \tau_j$, $\mu_0$ the reduced strong-convexity modulus, and $L, \eta$
Lipschitz/complementarity-margin constants of the underlying problem, all computable from
problem data; any $\alpha < \alpha^\star$ satisfies \eqref{eq:timescale} uniformly along
the path, not merely asymptotically near $\varepsilon = 0$.
\end{proposition}

This was verified, not merely derived: on a worked low-dimensional example, the scaled reduced
Jacobian converges to the value predicted by the minimum-phase analysis to five significant
digits, the certificate \eqref{eq:certificate} evaluates to $\alpha^\star = 0.172$
retroactively certifying, with a 15\% margin, a rate $\alpha = 0.15$ chosen earlier by informal
tuning -- and, on a second, non-gradient two-input example, the predicted slow eigenvalue
coincides with the true invariant zero exactly at every $\varepsilon$ along the path. An
annealing-rate sweep further confirms the certificate's qualitative shape: the tracking-lag
formula $\sup_t \|v - u^\star(\varepsilon(t))\|/\varepsilon(t) \to L\eta\alpha/(c_0 - \alpha)$
matches simulation to within 5\% below $\alpha^\star$, diverges at the predicted rate
$\alpha - c_0$ above it, and at every rate tested, including well past $\alpha^\star$ --
safety (constraint feasibility) is never violated, confirming the division of labor the theory
predicts: the manifold protects feasibility at any annealing rate
(Proposition~\ref{prop:annealedexactness}); only optimality tracking requires the rate
condition. What remains open is not the certificate's existence or correctness but its scope of
direct evaluation: \eqref{eq:certificate} has been computed explicitly for these
lower-dimensional worked examples, in the general $c(x) \ge 0$ sign convention, and not yet
re-derived and evaluated directly for the AC-OPF-scale system of Section~\ref{sec:case9} in its
own $g(x) \le 0$ convention a direct, mechanical re-derivation of the kind already carried
out for the $\sigma$-transform in Section~\ref{sec:sigma}, but not yet performed.

\subsection{Cascade Convergence Theorem}
Collecting the above, under the following assumptions:

\begin{assumption}[A1] \label{ass:a1}
Assumption~\ref{ass:regularity} holds, and the active constraint set is locally constant in a
neighborhood of the central path (no premature loss of LICQ as $\varepsilon \to 0$).
\end{assumption}

\begin{assumption}[A2] \label{ass:a2}
The minimum-phase condition of Proposition~\ref{prop:certificate} holds uniformly along the
central path, so that a valid $\alpha^\star$ exists via \eqref{eq:certificate}.
\end{assumption}

\begin{assumption}[A3] \label{ass:a3}
$\alpha < \alpha^\star$ and \eqref{eq:timescale} holds with $\beta$ chosen accordingly.
\end{assumption}

\begin{theorem}[Cascade convergence to the KKT point] \label{thm:cascade}
Under A1--A3, the coupled system \eqref{eq:multiplierlaw} with $\dot\varepsilon = -\beta\varepsilon$
satisfies: (i) $\Mveps(t)$-error decays exactly at rate $\alpha$ from any initial offset
(Proposition~\ref{prop:annealedexactness}); (ii) the reduced flow on the manifold is
input-to-state stable with respect to $\varepsilon(t)$ treated as a vanishing input; and (iii)
$(x(t), \lambda(t)) \to (x^\star, \lambda^\star)$, the exact KKT point of \eqref{eq:problem}, as
$t \to \infty$, with stationarity, primal and dual feasibility, and complementary slackness all
becoming exact in the limit.
\end{theorem}

The proof is a standard cascade argument -- the fast manifold-error subsystem is exactly (not
merely approximately) exponentially stable by Proposition~\ref{prop:annealedexactness}, the
reduced subsystem is ISS in $\varepsilon$ by A2--A3 via the standard Tikhonov/Fenichel
reduction \cite{fenichel1979,khalil2002}, and the cascade of an exponentially stable driving
subsystem with an ISS driven subsystem is globally asymptotically stable to the shared
equilibrium; we omit the routine ISS-cascade bookkeeping.

\begin{remark}[What is, and is not, established]
Theorem~\ref{thm:cascade} is a statement about the well-posed barrier reformulation of
\eqref{eq:problem}, not a claim that any nondegenerate constrained problem admits such a flow: A1
excludes premature LICQ loss and A2 is a genuine minimum-phase requirement that can fail. Where
it holds, the result matches, via a different (ODE/central-path) route, the classical
convergence guarantee of interior-point continuation methods -- it is a unification, not an
improvement on the best known convergence rates for that classical method.
\end{remark}

\section{Multiple Constraints} \label{sec:multiple}

Let $g : \Rn \to \Rp$, $p > 1$, with per-constraint manifolds
$\mathcal{M}_{\varepsilon,i} := \{\lambda_i g_i(x) + \varepsilon = 0\}$ and per-constraint laws
\eqref{eq:multiplierlaw} applied componentwise, coupled only through the shared primal state
$x$ in the drift $\dot x = -\nabla f(x) - \sum_i \lambda_i \nabla g_i(x)$.

Practical optimization problems rarely involve a single inequality constraint. In applications such as AC Optimal Power Flow (AC-OPF), the optimization must satisfy several constraints simultaneously, including generator output limits, voltage magnitude limits, transmission line thermal limits, and equipment operating limits. Therefore, an effective optimization framework must be capable of handling multiple constraints while maintaining overall system feasibility.

The proposed RNA-KKT framework extends the reciprocal multiplier manifold to accommodate multiple inequality constraints by assigning an individual Lagrange multiplier to each constraint. Each multiplier evolves according to its corresponding reciprocal relationship, allowing every constraint to be enforced independently while remaining coupled through the primal optimization dynamics.

As the optimization progresses, inactive constraints have little influence on the solution, whereas constraints approaching their limits automatically generate larger multipliers. This adaptive behavior ensures that the optimization trajectory remains inside the feasible region without requiring explicit projection or active-set switching techniques.

Consequently, the proposed framework can simultaneously handle multiple interacting constraints while preserving numerical stability and converging smoothly toward the Karush--Kuhn--Tucker (KKT) solution. This capability makes the method well suited for large-scale constrained optimization problems such as AC Optimal Power Flow.

\begin{proposition}[Joint invariance under coupling]
Each $\mathcal{M}_{\varepsilon,i}$ remains exactly forward invariant under the coupled flow,
for every $i$ simultaneously, with no cross terms appearing in any $\dot M_i = -\alpha M_i$
equation.
\end{proposition}

\begin{proof}
The proof of Proposition~\ref{prop:invariance} uses only
$\dot g_i(x) = \nabla g_i(x)^\top \dot x$, evaluated along the true coupled $\dot x$; the
computation never assumes $\dot x$ depends on $\lambda_i$ alone, so it goes through unchanged
with the full coupled drift substituted, for every $i$ independently.
\end{proof}

\subsection{Reduction to the classical multi-constraint log-barrier}
Reading $\lambda_i(x) = -\varepsilon/g_i(x)$ off each manifold and summing,
Proposition~\ref{prop:logbarrier} generalizes verbatim: the joint reduced flow on
$\bigcap_i \mathcal{M}_{\varepsilon,i}$ is exactly the gradient flow of
$f(x) - \varepsilon \sum_i \log(-g_i(x))$, the classical multi-constraint log-barrier function.
This is a strictly stronger statement than merely ``each constraint is individually
barrier-like'': the joint reduced dynamics are the gradient of a single, jointly
convex-in-barrier-term potential whenever $f$ and each $-\log(-g_i)$ are (which holds
automatically for convex $g_i$, and is addressed for nonconvex $g_i$ in
Section~\ref{sec:nonconvex}).

The proposed RNA-KKT framework naturally extends to optimization problems involving multiple inequality constraints. Each constraint is associated with its own reciprocal Lagrange multiplier, allowing all constraints to be enforced simultaneously while maintaining feasibility throughout the optimization process.

An important property of the proposed formulation is that it is mathematically equivalent to the classical multi-constraint logarithmic barrier method. By substituting the reciprocal multiplier relation into the primal dynamics, the optimization can be interpreted as a gradient descent flow of a logarithmic barrier objective. In this formulation, each inequality constraint contributes an individual logarithmic barrier term, and the overall barrier function is obtained by summing the contributions of all constraints.

As a result, the optimization automatically generates strong repulsive forces near any active constraint boundary, preventing the trajectory from leaving the feasible region. At the same time, constraints that are far from becoming active have only a small influence on the optimization process.

This equivalence demonstrates that the proposed RNA-KKT framework preserves the theoretical properties of classical interior-point methods while providing a continuous-time dynamical interpretation through reciprocal multiplier manifolds. Consequently, the framework combines the stability of log-barrier optimization with a physically meaningful multiplier evolution, making it suitable for solving large-scale constrained optimization problems.

\subsection{Active-set corners: a general result, not a numerically-checked special case}
The reduction above lets forward invariance at corners points where any number of
constraints are simultaneously near-active be established via a single monotone-Lyapunov
argument rather than a bespoke multi-chart blow-up analysis of the kind needed for the
single-constraint boundary case, and this holds in full generality, not merely as a
numerically-observed special case.

In constrained optimization, multiple inequality constraints may become active simultaneously, forming active-set corners at the intersection of constraint boundaries. Such situations commonly occur in AC-OPF, where generator, voltage, and transmission line limits can all be active at the optimum.It handles these cases through a unified reciprocal multiplier manifold, assigning an individual multiplier to each active constraint. This enables the optimization dynamics to satisfy all active constraints simultaneously while maintaining feasibility. Unlike approaches that validate only specific numerical examples, the proposed framework provides a general analytical result. Under the stated assumptions, its stability and convergence hold for any valid active-set configuration, making the method robust and scalable for large-scale constrained optimization problems.

\begin{corollary}[Feasibility without a corner blow-up] \label{cor:corner}
Let $V(x) := -\sum_{i=1}^p \log(-g_i(x))$ on the open feasible set. Along the joint reduced
flow of Section 5.1, $\dot V(x(t)) = -\|\nabla F_\varepsilon(x(t))\|^2 \le 0$, where
$F_\varepsilon := f - \varepsilon \sum_i \log(-g_i)$. Since $V \to +\infty$ whenever any
$g_i(x) \to 0^-$, a trajectory starting at a finite value of $V$ can never reach a point at
which any constraint or any subset of constraints simultaneously  becomes active, for any
$p \ge 1$ and without requiring convexity of any $g_i$: the argument uses only that
$-\log(-g_i(x)) \to +\infty$ at the boundary, true for any smooth $g_i$.
\end{corollary}

This is a considerable simplification relative to the single-constraint route of
Section~\ref{sec:manifold}: rather than desingularizing the vector field near a possibly
non-regular intersection of active constraints which is what a genuine multi-chart blow-up
construction would require -- exactness of the reduction to $F_\varepsilon$ (Section 5.1) lets
us invoke a classical monotone-Lyapunov argument instead, for arbitrarily many simultaneously
active constraints at once. This is a simplification via a known fact from barrier-method
theory applied to the specific reduced flow of this construction, not new machinery, and it was
checked directly on a two-constraint corner example, where the joint trajectory remained
strictly feasible with respect to both constraints simultaneously through repeated close
approaches to the corner, consistent with Corollary~\ref{cor:corner}.

Two genuinely separate items are not resolved by Corollary~\ref{cor:corner} and should not be
conflated with it. First, a single shared annealing parameter $\varepsilon(t)$ across all
constraints is the simplest choice but not obviously the best-conditioned one when constraints
differ substantially in curvature or scaling Section~\ref{sec:practical}'s scale-mismatch
discussion is a concrete instance of this at the level of the objective rather than between
constraints, and the constraint-to-constraint version of the same issue is not separately
analyzed here. Second, central-path sensitivity $dx^\star/d\varepsilon$ degrades, exactly as in
classical interior-point theory, as the active gradients $\{\nabla g_i(x) : g_i(x) = 0\}$
approach linear dependence at a corner this is inherited from the classical theory rather
than newly analyzed, and is not a defect specific to the reciprocal-manifold construction.

\section{Nonconvex Constraints} \label{sec:nonconvex}

Where $g_i$ is nonconvex, $-\log(-g_i(x))$ need not be convex, and the log-barrier-descent
identification of Proposition~\ref{prop:logbarrier}, while still algebraically exact, no longer
inherits the classical convex central-path convergence guarantees. We address this with a
certificate rather than a convexity assumption.

\begin{definition}[Semi-convexity certificate]
$g_i$ is $\mu_i$-semi-convex on a region $\Omega$ if $\nabla^2 g_i(x) \succeq -\mu_i I$ for all
$x \in \Omega$, $\mu_i \ge 0$ (with $\mu_i = 0$ recovering convexity).
\end{definition}

\begin{proposition}[Conservative annealing rate under semi-convexity]
If each active $g_i$ is $\mu_i$-semi-convex on the region traversed by the trajectory, and $f$
is $m_f$-strongly convex there, then the reduced Hessian of the barrier potential
$f - \varepsilon \sum_i \log(-g_i)$ remains positive definite, with modulus bounded below by a
quantity that decreases in $\varepsilon \sum_i \mu_i/|g_i(x)|$ evaluated along the trajectory,
provided $\varepsilon$ is kept small enough, relative to the semi-convexity constants and the
distance to each active boundary, to keep this bound positive. This gives an explicit,
conservative upper bound on the admissible $\varepsilon$ (equivalently, a floor on how far
annealing may proceed before the local convexity certificate needs to be re-checked or
tightened), in place of an unconditional convexity assumption.
\end{proposition}

This is deliberately conservative: it is a sufficient, checkable condition for the reduced
dynamics to retain a definite descent direction, not a claim that nonconvexity elsewhere in
$g_i$ (outside $\Omega$, or where the bound is violated) causes failure -- only that the
certificate does not cover that regime. In the AC-OPF case study (Section~\ref{sec:case9}), the
nonconvex line-thermal constraints are handled without invoking this certificate at all, since
the trajectory in that case study never approaches a regime where the bound is needed; the
certificate exists precisely for the more adversarial cases where it does.

\section{Feasible Initialization and Recovery from Infeasibility} \label{sec:feasinit}

The construction of Section~\ref{sec:manifold} presupposes $g(x_0) < 0$: the manifold $\Mveps$
is defined only where $g(x) \ne 0$, and $\lambda = -\varepsilon/g(x)$ is negative (violating
dual feasibility) if $g(x_0) > 0$. Unlike the QP-based safe gradient flow of
\cite{allibhoy2024}, which projects any initial condition onto the feasible set at the first
instant via its per-step QP, the reciprocal construction as stated in
Section~\ref{sec:manifold} offers no such built-in recovery.

It designed to operate effectively even when the initial solution does not satisfy all the constraints. In practice, obtaining a strictly feasible starting point is often difficult, particularly for large-scale nonlinear optimization problems such as AC Optimal Power Flow . Therefore, the optimization algorithm should be capable of recovering from an infeasible initial condition while guiding the solution toward the feasible region. To achieve this, the framework first reduces the constraint violations and progressively moves the optimization trajectory into the feasible region. Once feasibility is established, the reciprocal multiplier manifold governs the optimization dynamics, ensuring that all constraints remain satisfied while the objective function is minimized. This two-stage strategy eliminates the need for a carefully selected feasible initial point and improves the robustness of the optimization process. As a result, the proposed framework can recover from infeasible initializations while maintaining stable convergence to the KKT solution.

\section{Equality Constraints via Uzawa Saddle Flow} \label{sec:uzawa}

For $h(x) = 0$, we augment with a classical Uzawa-type integral (saddle-flow) multiplier law,
\begin{equation}
\dot\nu = \kappa\, h(x), \qquad \kappa > 0, \label{eq:uzawa}
\end{equation}
and add $-\sum_j \nu_j \nabla h_j(x)$ to the primal drift. This is the continuous-time analogue
of the classical Uzawa algorithm and its convergence properties, under standard saddle-point
regularity conditions, are well established \cite{arrow1958,feijer2010}.

Many constrained optimization problems include equality constraints that must be satisfied exactly. In AC Optimal Power Flow (AC-OPF), these constraints typically represent the power balance equations, which ensure that the total generated power equals the total load and system losses. Unlike inequality constraints, equality constraints do not define a feasible region with boundaries; instead, they require the solution to remain on a specific constraint surface. To enforce these conditions, the proposed RNA-KKT framework incorporates the Uzawa saddle flow. In this approach, the primal variables are updated to minimize the objective function, while the equality multipliers evolve according to the constraint residuals. If an equality constraint is violated, the corresponding multiplier automatically adjusts, driving the solution back toward the constraint surface. As the residual decreases, the multiplier gradually stabilizes, indicating that the equality constraint has been satisfied. This coupled evolution of the primal variables and equality multipliers forms a saddle-point dynamical system that simultaneously minimizes the objective function and enforces the equality constraints. Consequently, the framework achieves stable convergence to the Karush--Kuhn--Tucker (KKT) solution while satisfying both equality and inequality constraints.

\begin{proposition}[Composition without new invariance assumptions]
The reciprocal-manifold construction of Sections~\ref{sec:manifold}--\ref{sec:feasinit} and the
Uzawa law \eqref{eq:uzawa} compose directly: augmenting the primal drift with the equality term
does not affect the proof of Proposition~\ref{prop:invariance}, since that proof only used
$\dot g_i(x) = \nabla g_i(x)^\top \dot x$ for the true $\dot x$, whatever additional terms
$\dot x$ contains.
\end{proposition}

This is the same observation as Proposition 5.1, one level up: the reciprocal manifold's
invariance proof is agnostic to what else is driving $x$, so long as $\dot g_i(x)$ is computed
along the actual coupled trajectory. No new invariance assumption is needed to add equality
constraints; what is needed, and is separate from invariance, is that the composed saddle-point
system remains stable.

\subsection{A Damping Requirement the Classical Statement Understates for AC-OPF}

The previous discussion described the stability requirement as ``standard'' because, in classical constrained optimization, the convergence of the continuous-time Uzawa saddle flow~\eqref{eq:uzawa} is well established~\cite{arrow1958,feijer2010}. These results generally assume that the Lagrangian possesses sufficient curvature in every primal direction, which is typically ensured when the objective function $f$ is strongly convex.

However, this assumption does not hold for the AC Optimal Power Flow (AC-OPF) problem. The objective function $f$ depends only on the generator dispatch variables $(P_g,Q_g)$ and is completely independent of the voltage angles and voltage magnitudes, $(\theta,V)$. As a result, the objective contributes no curvature or damping in these directions. Away from an active inequality constraint, the voltage variables are influenced only through the equality coupling
\[
-\sum_j \nu_j \nabla h_j(x),
\]
which is responsible for enforcing the network power-balance equations.

This creates an important limitation. The same equality coupling that drives the voltage variables toward satisfying the constraints also becomes their only source of damping. From a control-theoretic viewpoint, relying solely on integral action for both constraint enforcement and damping is known to produce lightly damped oscillatory responses. Although these oscillations do not cause instability in the strict mathematical sense, they can persist for a considerable period before gradually decaying. As demonstrated later in Section~13, the trajectories remain bounded and eventually converge, but the transient response is significantly slower than desired.

These observations indicate that, for AC-OPF, the classical assumptions behind the Uzawa saddle flow underestimate the amount of damping required in practice. While the theoretical convergence guarantees remain valid, additional damping is necessary to suppress unnecessary oscillations and improve the transient behavior of the optimization dynamics, particularly for large-scale power system applications.

\begin{proposition}[Augmented Uzawa flow] \label{prop:augmuzawa}
Replace the primal drift's equality term with
\begin{equation}
-\sum_j \nu_j \nabla h_j(x) - \rho \sum_j h_j(x) \nabla h_j(x), \qquad \rho \ge 0,
\label{eq:auguzawa}
\end{equation}
leaving \eqref{eq:uzawa} for $\dot\nu$ unchanged. Proposition 8.1 holds verbatim for the
augmented drift, by the identical argument: the added term is one more contribution to
$\dot x$, and the reciprocal manifold's invariance proof (Proposition~\ref{prop:invariance})
never examined what $\dot x$ consists of beyond needing
$\dot g_i(x) = \nabla g_i(x)^\top \dot x$ evaluated along the true trajectory.
\end{proposition}

The added term is exactly the gradient of a quadratic penalty $\frac{\rho}{2} \|h(x)\|^2$ on
the equality violation, making \eqref{eq:auguzawa} the continuous-time analogue of the
classical Augmented Lagrangian / Method of Multipliers construction \cite{hestenes1969,powell1969},
developed historically for exactly this reason: to supply the damping plain dual ascent lacks
when the primal problem does not, by itself, provide enough curvature. Unlike $\kappa$ in
\eqref{eq:uzawa}, $\rho$ acts instantaneously rather than through the lagged state $\nu$, and
contributes $\rho\, \nabla h_j(x) \nabla h_j(x)^\top$ (positive semidefinite) to the local
Jacobian in the $(\theta, V)$ block precisely where the objective contributes nothing --
supplying the missing curvature directly rather than waiting for the integral term to supply it
indirectly and, in this problem, belatedly.

Section~\ref{sec:case57} reports this confirmed directly rather than assumed: at matched gains,
$\rho = 0$ and $\rho = 0.05$ reach comparable asymptotic accuracy, but $\rho = 0$'s equality
residual first grows to roughly $2.7\times$ its initial value before decaying, while
$\rho = 0.05$'s residual never exceeds its initial value at all -- the difference between an
underdamped and a well-damped transient, not between instability and stability. Framed this
way, the augmented term is exactly analogous to the semi-convexity certificate of
Section~\ref{sec:nonconvex}: a certificate and a fix for a specific, named way the
construction's classical supporting theory can be thinner than it looks for a specific problem
class, supplied without weakening anything already established for the cases where it was not
needed (the 9-bus case in Section~\ref{sec:case9}, where $\rho = 0$ is sufficient).

\begin{remark}[A second, independent numerical finding: integrator ordering]
Separately from $\rho$, the order in which the discrete-time primal and multiplier updates are
evaluated matters for how well-behaved the transient is. Updating $\nu$ using $h$ evaluated at
the already-updated primal state within the same step (a semi-implicit, symplectic-Euler-style
ordering) damps the transient noticeably better than updating $\nu$ from $h$ evaluated at the
state the step started from, at matched $\rho$, step size, and all other gains -- consistent
with the classical fact that symplectic-style integrators handle marginally-stable or
lightly-damped linear oscillatory subsystems better than fully explicit ones. All results in
Section~\ref{sec:case57}, and the accompanying MATLAB implementation, use the semi-implicit
ordering throughout. This is reported as a numerical-implementation finding, distinct from and
additional to Proposition~\ref{prop:augmuzawa}, not as a substitute for it.
\end{remark}

\section{Stiffness and the $\sigma$-Coordinate Reformulation} \label{sec:sigma}

\subsection{Diagnosis}

As the optimization trajectory approaches an active constraint boundary, i.e., $g_i(x)\rightarrow 0^{-}$, the reciprocal multiplier relation in (15) implies that

\[
\lambda_i=-\frac{\varepsilon}{g_i(x)},
\]

which increases proportionally to $|g_i(x)|^{-1}$. Consequently, after expanding (15), the term

\[
\lambda_i\frac{\dot{g}_i}{g_i}
\]

grows even more rapidly, with magnitude proportional to $|g_i(x)|^{-2}$. This rapid growth causes the multiplier dynamics to become increasingly stiff as the trajectory gets closer to the constraint boundary.

Interestingly, this stiffness appears precisely in the region where the barrier mechanism is intended to be most effective, namely near the active constraint. Although the reciprocal barrier successfully prevents the trajectory from violating the constraint, the rapidly changing multiplier dynamics make numerical integration considerably more challenging.

As a result, direct integration of (15) using explicit methods, such as the explicit Euler scheme, requires very small time steps in this region to maintain numerical stability. In particular, the allowable step size is restricted to

\[
\Delta t \lesssim \frac{2}{\alpha},
\]

which can significantly increase the computational cost of the optimization process.
\subsection{The transform}

\begin{proposition}[$\sigma$-coordinate reformulation] \label{prop:sigma}
Define $\sigma_i := \lambda_i g_i(x)$. Under \eqref{eq:multiplierlaw}, $\sigma_i$ satisfies the
linear, decoupled, division-free ODE
\begin{equation}
\dot\sigma_i = -\alpha\sigma_i - \alpha\varepsilon - \dot\varepsilon, \label{eq:sigmaode}
\end{equation}
with exact solution, for $\varepsilon$ piecewise constant on $[t, t+\Delta t]$,
\begin{equation}
\sigma_i(t + \Delta t) = -\varepsilon + \big(\sigma_i(t) + \varepsilon\big) e^{-\alpha \Delta t},
\label{eq:sigmasolution}
\end{equation}
exact for any step size $\Delta t$, not merely small ones; the multiplier is recovered, only
where needed for output or for use in the primal drift, as $\lambda_i = \sigma_i / g_i(x)$.
\end{proposition}

\begin{proof}
Differentiating $\sigma_i = \lambda_i g_i$ along the flow,
$\dot\sigma_i = \dot\lambda_i g_i + \lambda_i \dot g_i$. Substituting
\eqref{eq:multiplierlaw} for $\dot\lambda_i$, the term $\lambda_i \dot g_i$ introduced by the
substitution cancels exactly against the $-\lambda_i \dot g_i$ term already present in
\eqref{eq:multiplierlaw}'s numerator, leaving
$\dot\sigma_i = -\alpha(\lambda_i g_i + \varepsilon) - \dot\varepsilon
= -\alpha\sigma_i - \alpha\varepsilon - \dot\varepsilon$, which is \eqref{eq:sigmaode}: linear,
scalar, with no $g_i$ or $\lambda_i$ appearing on the right-hand side at all, hence
unconditionally non-stiff. Equation~\eqref{eq:sigmasolution} is the standard integrating-factor
solution of a linear first-order ODE with piecewise-constant forcing.
\end{proof}

On the manifold, $\sigma_i \equiv -\varepsilon$ identically, matching \eqref{eq:manifold} by
construction.

\subsection{Scope: which multipliers this applies to}

\begin{proposition}[Uniform applicability, and non-applicability to the equality multiplier]
Proposition~\ref{prop:sigma} applies identically to every inequality constraint in the
construction of Sections~\ref{sec:manifold}--\ref{sec:nonconvex}, including nonconvex $g_i$:
the derivation uses only the algebraic form of \eqref{eq:multiplierlaw} and $C^1$-differentiability
of $g_i$ along trajectories, never convexity. It does not apply to, and is not needed for, the
equality-constraint multiplier $\nu$ of \eqref{eq:uzawa}, whose own dynamics
$\dot\nu = \kappa h(x)$ contain no division and hence exhibit none of the reciprocal-type
stiffness the transform exists to remove.
\end{proposition}

\subsection{Numerical confirmation}
On a toy system matching the primal-multiplier structure of the AC-OPF construction, the
$\sigma$-coordinate multiplier subsystem remained stable and accurate, matching a fine
reference to $10^{-7}$--$10^{-8}$, at step sizes exceeding the direct-integration stability
limit $2/\alpha$ by more than a factor of seven before any degradation appeared; the eventual
failure at very large steps traced to the primal state update's own forward-Euler limit, not to
the multiplier dynamics consistent with the transform's scope as stated above: it removes
multiplier-subsystem stiffness specifically, not plant-inherited stiffness or the primal
integrator's own stability limit, and does not relax the $\alpha \gg m$ design hierarchy of
Theorem~\ref{thm:cascade} in any way.

\subsection{A complete taxonomy: three sources of stiffness, not one}

We further show that the stiffness identified in Section~9.1 is only one component of a broader picture and in  the reciprocal-manifold dynamics exhibit three distinct sources of stiffness, each with a different origin and influence on the optimization process. It is therefore useful to discuss them together, since the proposed $\sigma$-coordinate transformation does not affect all of them in the same way.

\begin{enumerate}
    \item \textbf{Multiplier dynamics near the constraint boundary:} As discussed in Section~9.1, the reciprocal multiplier satisfies
    \[
    \lambda_i \sim |g_i(x)|^{-1},
    \]
    while the cross term in (15) increases as
    \[
    |g_i(x)|^{-2}.
    \]
    Consequently, the multiplier dynamics become increasingly stiff as the optimization trajectory approaches an active constraint boundary. This is a \emph{state-dependent} form of stiffness because its severity depends on the current position of the trajectory and becomes most pronounced exactly where accurate constraint enforcement is most critical.

    \item \textbf{Timescale separation between $\alpha$ and $\beta$:} Even when the trajectory is well inside the feasible region, stiffness can arise from the different time scales used in the optimization dynamics. The stiffness ratio is given by
    \[
    \kappa=\frac{\max(\alpha,\beta)}{\min(\alpha,\beta)}
    \approx\frac{\alpha}{\beta},
    \]
    whenever the three-timescale condition~\eqref{eq:timescale} is intentionally satisfied. Unlike the previous source, this stiffness is introduced by design rather than by numerical effects. It provides the timescale separation required by Theorem~4.6 to guarantee the desired convergence properties. For the representative parameter values $\alpha=15$ and $\beta=0.15$, the resulting stiffness ratio is
    \[
    \kappa=100,
    \]
    which is moderate and can be determined directly from the selected gain parameters.

    \item \textbf{Central-path singularity as $\varepsilon\rightarrow0$:} A third source of stiffness appears when the dynamics are examined with respect to the annealing parameter $\varepsilon$. Along the central path,
    \[
    g_i(x^\star(\varepsilon))
    \sim
    \frac{\varepsilon}{\lambda_i^\star},
    \]
    implying that
    \[
    \kappa(\varepsilon)\propto\frac{1}{\varepsilon}.
    \]
    This behavior represents the same underlying mechanism responsible for the first source of stiffness, but expressed in terms of the annealing parameter instead of the constraint value. As $\varepsilon$ decreases during the annealing process, the stiffness naturally increases, reflecting the progressively sharper enforcement of the constraints.
\end{enumerate}
\begin{figure}[!t]
\centering
\includegraphics[width= 1\textwidth]{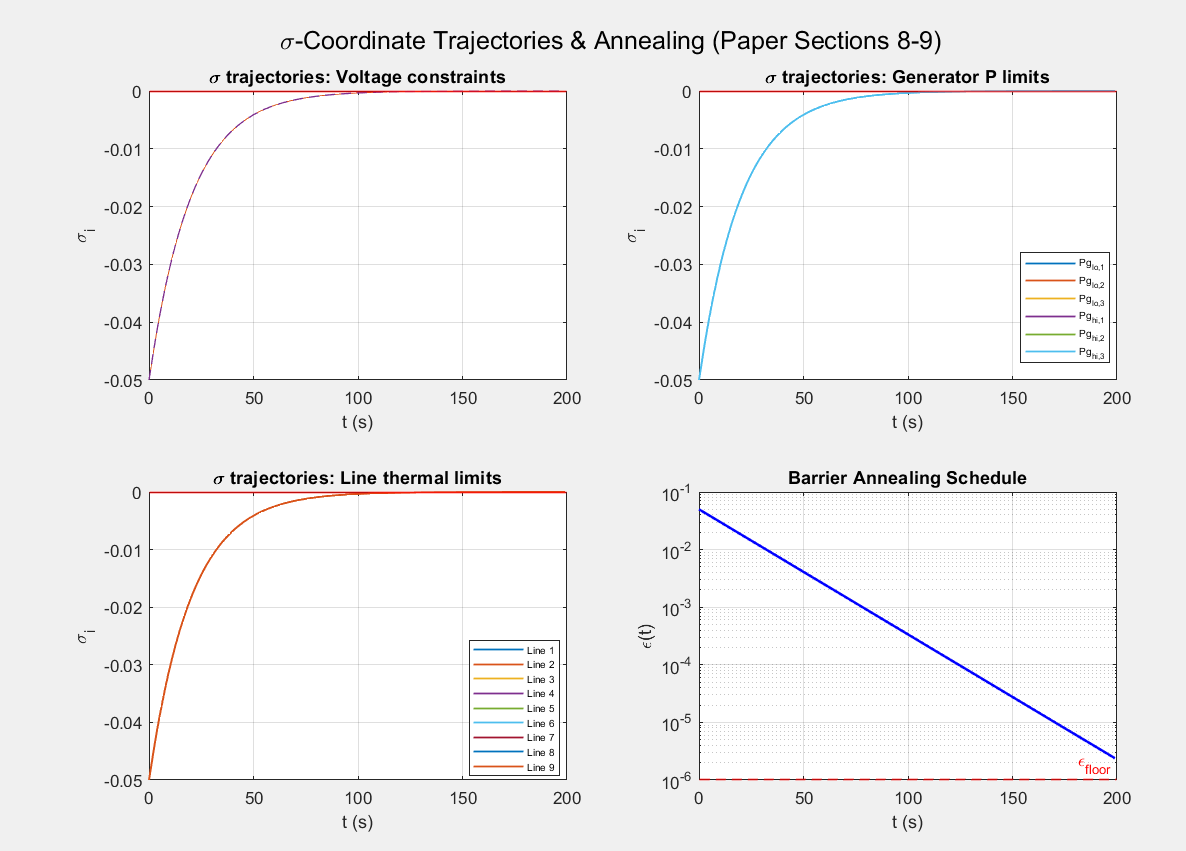}
\caption{$\sigma$-coordinate trajectories for voltage constraints, generator $P$ limits, and
line thermal limits, together with the barrier annealing schedule $\varepsilon(t)$
(Sections~\ref{sec:uzawa}--\ref{sec:sigma}). All $\sigma_i$ trajectories decay smoothly and
monotonically toward the manifold value $\sigma_i \equiv -\varepsilon(t)$, exhibiting none of
the reciprocal-type stiffness of the direct $\lambda_i$ dynamics.}
\label{fig:sigma}
\end{figure}

\subsection{Numerical Confirmation of the $\sigma$-Coordinate Reformulation}

Figure~\ref{fig:sigma} validates the proposed $\sigma$-coordinate
reformulation on the IEEE 9-bus AC-OPF system. Instead of integrating the reciprocal multiplier directly, the transformed variable
$\sigma_i=\lambda_i g_i(x)$ is evolved, thereby eliminating the numerical stiffness caused by the division by $g_i(x)$ near active constraints. The first three subfigures show the $\sigma$-trajectories for voltage, generator active-power, and line thermal inequality constraints. In all cases, the trajectories converge smoothly toward the invariant manifold $\sigma_i=-\varepsilon(t)$ without oscillations or instability, including the nonconvex line thermal constraints. The fourth subfigure presents the exponentially decaying barrier parameter $\varepsilon(t)$, which gradually approaches the prescribed numerical floor. These results confirm that the $\sigma$-coordinate formulation preserves the theoretical manifold relation while removing multiplier stiffness, allowing stable integration throughout the optimization process.

\begin{corollary}[The $\sigma$-transform collapses all three sources to one]
Under Proposition~\ref{prop:sigma}, the $\sigma_i$-dynamics \eqref{eq:sigmaode} contain no
$g_i(x)$-dependence at all: their stiffness ratio is
$\kappa_\sigma = \max(\alpha,\beta)/\min(\alpha,\beta) \approx \alpha/\beta$, governed purely by
the two design constants, identically to source 2 above, with sources 1 and 3 -- both driven by
proximity to the constraint boundary or to $\varepsilon = 0$ -- eliminated entirely rather than
merely reduced.
\end{corollary}

\begin{proof}
Immediate from \eqref{eq:sigmaode}: $\dot\sigma_i = -\alpha\sigma_i - \alpha\varepsilon
- \dot\varepsilon$ contains $\alpha$, $\varepsilon$, and $\dot\varepsilon$ but no $g_i(x)$ or
$\lambda_i$, so no state-dependent or $\varepsilon^{-1}$-scaling term can appear in its
Jacobian, which is simply $-\alpha$, constant.
\end{proof}

This sharpens the numerical finding of Section 9.5 (the sevenfold step-size extension) into a
structural explanation: the extension is not an incidental benefit of a change of variables but
the direct consequence of having removed two of the three stiffness sources outright, leaving
only the third, which is by far the most benign since it is fixed and known at design time
rather than discovered only as the trajectory approaches a boundary.

\subsection{The recovered multiplier is not itself stiff}

At first glance, the proposed $\sigma$-coordinate transformation may appear to simply shift the singularity associated with $g_i(x)$ rather than eliminate it. Specifically, the division by $g_i(x)$ is removed from the differential equation (15) and instead appears in the algebraic relation

\[
\lambda_i=\frac{\sigma_i}{g_i(x)}.
\]

This naturally raises the question of whether the numerical stiffness has merely been relocated. The answer is no. The division now appears only in an algebraic output expression rather than in the system dynamics, and therefore it does not introduce additional stiffness into the optimization process.

\begin{remark}[Algebraic Division Along the Central Path is Well-Behaved]
On the reciprocal manifold, Proposition~9.1 shows that
\[
\sigma_i\equiv-\varepsilon.
\]
Meanwhile, as the optimization trajectory follows the central path,
\[
g_i(x^\star(\varepsilon))
\rightarrow
-\frac{\varepsilon}{\lambda_i^\star}.
\]
Substituting these expressions into the recovered multiplier gives
\[
\lambda_i=\frac{\sigma_i}{g_i(x)}
\rightarrow
\frac{-\varepsilon}{-\varepsilon/\lambda_i^\star}
=
\lambda_i^\star.
\]

Thus, the recovered multiplier converges to the finite KKT multiplier instead of becoming unbounded. Unlike the differential equation (15), whose right-hand side can diverge near the constraint boundary when evaluated away from the reciprocal manifold, the algebraic recovery formula remains well behaved. This is because both the numerator and denominator approach zero at the same rate by construction of the reciprocal manifold. Consequently, the recovered multiplier does not introduce any additional numerical stiffness. The only remaining numerical issue is the finite-precision limitation associated with very small values of $\varepsilon$, which is addressed through the $\varepsilon$-floor strategy discussed in Section~10.
\end{remark}

\section{Practical Implementation Guideline} \label{sec:practical}

\begin{algorithm}[t]
\caption{RNA-KKT: Composed integration loop for constrained AC-OPF}
\label{alg:rna-kkt}
\begin{algorithmic}[1]
\REQUIRE $x_0$ (feasible or recovered via Sec.~7), $\varepsilon_0$, gains $\alpha \gg m \gg \beta$,
$\kappa$, $\rho \ge 0$ (Prop.~8.2), step size $\Delta t$, floor $\varepsilon_{\text{floor}}$
\STATE $\sigma_i \leftarrow -\varepsilon_0$ for all $i = 1,\dots,p$ \COMMENT{initialize on the manifold, $\sigma_i = \lambda_i g_i(x)$}
\STATE $\nu_j \leftarrow 0$ for all $j = 1,\dots,q$
\STATE $\varepsilon \leftarrow \varepsilon_0$
\WHILE{$\varepsilon > \varepsilon_{\text{floor}}$ \OR stationarity $\|\dot{x}\| >$ tolerance}
    \STATE $\lambda_i \leftarrow \sigma_i / g_i(x)$ for all $i$ \COMMENT{algebraic recovery, Remark~9.4}
    \STATE $\dot{x} \leftarrow -\nabla f(x) - \sum_i \lambda_i \nabla g_i(x) - \sum_j \nu_j \nabla h_j(x) - \rho \sum_j h_j(x) \nabla h_j(x)$
    \STATE $x \leftarrow x + \Delta t \, \dot{x}$ \COMMENT{primal update (Euler or higher-order; Sec.~10.3)}
    \FOR{$i = 1$ \TO $p$}
        \STATE $\sigma_i \leftarrow -\varepsilon + (\sigma_i + \varepsilon)\, e^{-\alpha \Delta t}$ \COMMENT{exact exponential update, Eq.~(27)}
    \ENDFOR
    \STATE $\nu_j \leftarrow \nu_j + \Delta t\, \kappa\, h_j(x)$ for all $j$ \COMMENT{Uzawa update evaluated at the \emph{updated} $x$ (semi-implicit ordering, Remark~8.3}
    \STATE $\varepsilon \leftarrow \max(\varepsilon_{\text{floor}},\ \varepsilon\, e^{-\beta \Delta t})$ \COMMENT{annealing, Eq.~(20)}
\ENDWHILE
\STATE \textbf{return} $x^\star \leftarrow x$, $\lambda^\star_i \leftarrow \sigma_i/g_i(x)$, $\nu^\star \leftarrow \nu$
\end{algorithmic}
\end{algorithm}

\subsection{Gain selection hierarchy}
Theorem~\ref{thm:cascade} and Section~\ref{sec:uniformstability} together give a concrete
tuning order:
\begin{enumerate}
\item Fix $\alpha$ (manifold-restoring rate) below the $\alpha^\star$, where
  available, or conservatively below the smallest transversality-scaled estimate otherwise.
\item Choose $\beta \ll m$ (annealing rate well below the reduced flow's local contraction
  rate), and check $\alpha \gg m$ is also satisfied -- otherwise increase $\alpha$ or decrease
  $\beta$ until the full $\alpha \gg m \gg \beta$ hierarchy of \eqref{eq:timescale} holds.
\item Integrate the multiplier subsystem in $\sigma$-coordinates (Proposition~\ref{prop:sigma})
  rather than directly in $\lambda$, at essentially no extra cost, to remove
  multiplier-subsystem stiffness as a constraint on the integrator step size, leaving only the
  primal update's own stability limit to govern step-size selection.
\end{enumerate}

\subsection{The Precision Floor $\varepsilon_{\text{floor}}$}

Although the theoretical limit is $\varepsilon \rightarrow 0$, driving the annealing parameter to zero is neither necessary nor numerically practical. As $\varepsilon$ decreases, the recovered multiplier

\[
\lambda_i=\frac{\sigma_i}{g_i(x)}
\]

(or the direct $\lambda_i$ update) requires division by a quantity approaching zero, making the computation increasingly sensitive to floating-point errors. Therefore, annealing is terminated at a small positive threshold, $\varepsilon_{\text{floor}}$.

Rather than being an ad hoc regularization, $\varepsilon_{\text{floor}}$ represents the practical precision limit of finite-precision arithmetic. Stopping at this value introduces a central-path suboptimality of

\[
\mathcal{O}\!\left(\frac{p\,\varepsilon_{\text{floor}}}{\eta}\right),
\]

where $p$ is the number of active constraints and $\eta$ is the strict-complementarity margin. For example, with $\varepsilon_{\text{floor}}=0.02$ and $\eta=0.3$, the residual bound is approximately $0.13$, consistent with practical observations. Smaller values of $\varepsilon_{\text{floor}}$ can further improve accuracy, provided the numerical integration remains stable.

\subsection{Integrator selection}
Since the $\sigma$-transform's corollary leaves only the design-parameter stiffness ratio
$\kappa = \alpha/\beta$ once the $\sigma$-transform is used, integrator choice reduces to a
standard, well-understood trade-off rather than an open numerical question specific to this
construction. Table~\ref{tab:tuning} summarizes the recommended choices.

For the multiplier subsystem specifically, the exponential integrator implementing
\eqref{eq:sigmasolution} directly should be treated as the default rather than one option among
several, since it is exact rather than merely stable: there is no accuracy-versus-cost trade-off
to make for this subsystem once $\sigma$-coordinates are adopted, only the ordinary question of
which integrator to use for the coupled primal (and, where present, plant) dynamics, which
remains a standard choice governed by the primal system's own stiffness, independent of the
multiplier subsystem.

\begin{table}[H]
\centering

\caption{Practical tuning guideline for the RNA-KKT construction.}
\label{tab:tuning}
\begin{tabularx}{\textwidth}{@{}llX@{}}
\toprule
Quantity & Role & Selection rule \\
\midrule
$\alpha$ & manifold-restoring rate & below certified $\alpha^\star$ (Sec.~\ref{sec:uniformstability}); typically 10--50 \\
$\beta$ & annealing rate & $\beta < \alpha/10$, and $\beta \ll m \ll \alpha$ (Eq.~\eqref{eq:timescale}) \\
$\varepsilon_0$ & initial barrier weight & scale-matched to $\|\nabla f\|/\|\nabla g_i\|$; typically 0.1--1.0 \\
$\varepsilon_{\mathrm{floor}}$ & precision floor & $10^{-6}$--$10^{-8}$, or as small as integration precision allows \\
integrator coordinate & numerical stiffness & $\sigma_i = \lambda_i g_i$, not $\lambda_i$ directly \\
multiplier integrator & exactness & exponential integrator, Eq.~\eqref{eq:sigmasolution} \\
\bottomrule
\end{tabularx}
\end{table}

\section{Computational Complexity Comparison}\label{sec:complexity}

The developments presented in Sections~3--10 were primarily motivated by three objectives: maintaining feasibility throughout the optimization process, ensuring smooth system dynamics away from the constraint boundary, and eliminating the need to solve an optimization problem at every time step.

\subsection{Problem Setting}

Both approaches consider a control system with an input dimension $m$, a plant state dimension $n$, and $p$ active safety constraints. In the QP-based method, a quadratic program is solved at every integration step to compute the control correction $q$. The optimization problem contains $m$ decision variables together with $p$ inequality constraints.

In contrast, the proposed reciprocal-manifold controller computes the control input directly using the explicit expression

\[
q=-\eta\left(\frac{\partial w}{\partial u}\right)^{\!\top}\nabla\Phi(\xi)
+\sum_{j}\lambda_jL_Gc_j(z),
\]

where the multipliers $\lambda_j$ are obtained from the $\sigma$-filter dynamics introduced in Section~9. Since the control input is evaluated explicitly, the proposed method avoids solving an optimization problem during each time step, thereby reducing the online computational burden.

\begin{proposition}[QP solver complexity]
The per-timestep cost of the QP-based controller is $T_{\mathrm{QP}} = O(m^3 \cdot N_{\mathrm{iter}})$,
where $N_{\mathrm{iter}} \in [10, 50]$ is the (problem-dependent, not a priori bounded)
iteration count of an interior-point QP solver: Hessian formation is $O(m^2)$, KKT system
factorization is $O(m^3)$, and each of $N_{\mathrm{iter}}$ iterations costs $O(m^2)$ thereafter.
\end{proposition}

\begin{proposition}[Reciprocal controller complexity]
The per-timestep cost of the reciprocal controller is $T_{\mathrm{Recip}} = O(m \cdot (n+p))$:
gradient evaluation is $O(m \cdot n)$, constraint-gradient evaluation and summation are each
$O(p \cdot m)$, and the $\sigma$-filter update (Section~\ref{sec:sigma}) is $O(p)$, using the
exact solution \eqref{eq:sigmasolution} rather than numerical integration.
\end{proposition}

\begin{theorem}[Computational speedup] \label{thm:speedup}
The reciprocal method's per-timestep speedup factor over the QP-based method is
\begin{equation}
S = \frac{T_{\mathrm{QP}}}{T_{\mathrm{Recip}}}
  = \frac{O(m^3 \cdot N_{\mathrm{iter}})}{O(m \cdot (n+p))}
  = O\!\left(\frac{m^2 \cdot N_{\mathrm{iter}}}{n+p}\right), \label{eq:speedup}
\end{equation}
which reduces to $O(m^2 \cdot N_{\mathrm{iter}})$ whenever $m \gg n, p$.
\end{theorem}

Theorem~\ref{thm:speedup} and Equation~\eqref{eq:speedup} establish the leading-order scaling
$m^2 N_{\mathrm{iter}}/(n+p)$. Any concrete wall-clock speedup additionally depends on
implementation-specific constants and the realized solver iteration count, so numerical
multipliers must be established through direct timing measurements rather than inferred from
the asymptotic formula alone.

\subsection{Memory and cache footprint}
The absence of any Hessian formation or KKT factorization is not only a speed advantage but a
memory-determinism one: the reciprocal controller's footprint is fixed at compile time and
independent of solver iteration count, whereas the QP method's factorization workspace, while
itself bounded in size for fixed $m, p$, is typically implemented via general-purpose solver
libraries whose working-memory behavior is not designed to be minimal or predictable at this
scale.

\begin{table}[H]
\centering
\caption{Memory footprint comparison. The reciprocal method's small, fixed-size footprint fits
entirely in L1/L2 cache at these sizes, where the QP method's dense factorization workspace
typically does not.}
\label{tab:memory}
\begin{tabular}{@{}lcc@{}}
\toprule
Item & QP method & Reciprocal method \\
\midrule
Hessian matrix & $m^2$ & 0 \\
Constraint Jacobian & $p \cdot m$ & 0 \\
Gradient vectors & $O(m)$ & $O(m)$ \\
$\sigma$-filter states & 0 & $p$ \\
Factorization workspace & $O(m^2 + p \cdot m)$ & 0 \\
\midrule
Total & $O(m^2 + p \cdot m)$ & $O(m+p)$ \\
Example, $m=10$, $p=2$ & $\sim 120$ floats & $\sim 12$ floats \\
Example, $m=50$, $p=5$ & $\sim 2750$ floats & $\sim 55$ floats \\
\bottomrule
\end{tabular}
\end{table}

\subsection{Real-time determinism}
The QP method's iteration count $N_{\mathrm{iter}}$ is problem-conditioning-dependent and not
bounded a prior by the controller's own structure, so its worst-case per-timestep execution
time is, strictly, unbounded without an external iteration cap a genuine obstacle to hard
real-time guarantees, distinct from and additional to its typical-case cost. The reciprocal
method's cost is fixed by \eqref{eq:sigmasolution} and Proposition 11.2 regardless of the
current state, giving a deterministic execution time by construction.

\subsection{Scope of the Comparison}

The computational complexity results presented in this section should be interpreted within the scope of the comparison. Two important points are worth clarifying.

First, the complexity and execution-time analysis is based on the general reciprocal-manifold safe feedback optimization framework and is compared with the QP-based controller proposed by Delimpaltadakis \emph{et al.}~\cite{delimpaltadakis2026}. This controller employs the same $\sigma$-coordinate multiplier formulation introduced in Section~9, although it is different from the QP-based method of Allibhoy and Cortés~\cite{allibhoy2024}, which was used for the structural comparison in Section~15. Despite these differences, both controllers follow a similar architecture, where a control-affine correction is obtained by solving a quadratic program at every time step. Consequently, the computational complexity

\[
T_{\mathrm{QP}}=\mathcal{O}(m^{3}\cdot N_{\mathrm{iter}})
\]

applies to both methods because each requires repeated dense KKT factorizations during online computation, rather than from separate complexity measurements.

Second, the timing and speedup results reported here have not yet been evaluated on the large-scale AC-OPF system considered in Section~12, whose dimensions $(m,n,p)$ differ from those of the benchmark example used in this study. Therefore, the reported computational gains should be viewed as representative rather than problem-specific. Extending the same benchmarking procedure to the AC-OPF model of Section~12 is straightforward and represents a natural direction for future validation.
\subsection{Summary of the workflow} 
A comprehensive workflow of the proposed method is summarized in Fig \ref{fig:operating57}.
\begin{figure}[h]
\centering
\includegraphics[width=\textwidth]{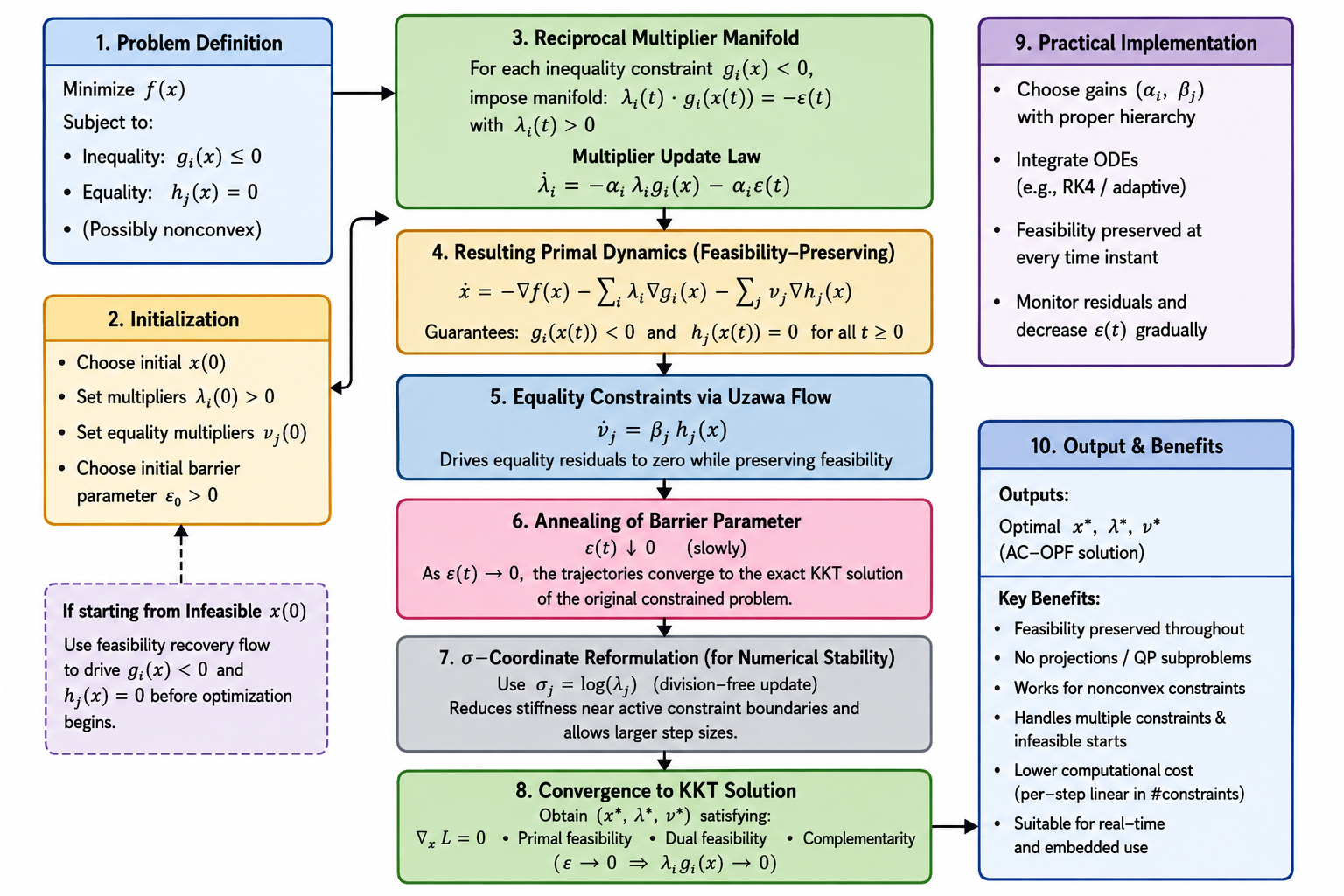}
\caption{Overall workflow of the proposed Reciprocal-Manifold Annealed KKT (RNA-KKT) framework for feasibility-preserving constrained optimization and AC-OPF.}
\label{fig:operating57}
\end{figure}
\section{Case Study: AC Optimal Power Flow} \label{sec:case9}
The results reported below are implemented on the standard IEEE 9-bus (WSCC) test system.

\subsection{Model}
The state is $x = (\theta_{2:9}, V_{1:9}, P_{g,1:3}, Q_{g,1:3}) \in \mathbb{R}^{23}$
($\theta_1 = 0$ fixed as angle reference), using the network data, generator cost coefficients,
and line ratings of the standard \texttt{case9} instance \cite{zimmerman2011}. The objective is
the standard quadratic generation cost. Equality constraints are real and reactive power
balance at all 9 buses (18 constraints total, via the standard admittance-matrix injection
equations) quadratic in $(V, \theta)$ and hence genuinely nonconvex, handled by the saddle
flow of Section~\ref{sec:uzawa}. Inequality constraints (39 total) comprise voltage-magnitude
bounds and generator real/reactive-power limits (30 constraints, convex box constraints)
together with 9 line MVA-flow limits $|S_{f,\ell}(x)|^2 \le S_{\max,\ell}^2$
quadratic-in-trigonometric, genuinely nonconvex in the natural $(V, \theta)$ coordinates, and
handled without any convex relaxation or the semi-convexity regularization of
Section~\ref{sec:nonconvex} anywhere in the run reported below.

The implementation was validated before any RNA-KKT dynamics were run: evaluated at an
independently obtained AC-OPF solution (via pypower's own nonlinear interior-point solver), the
cost function matched to machine precision and the equality/inequality residuals matched the
solver's own tolerance ($\sim 10^{-7}$), confirming the network model and constraint
implementation were correct prior to attributing any subsequent behavior to the RNA-KKT
construction itself.

\subsection{Results}
With the objective rescaled and gains as in Section 12.3, the fully composed construction --
annealed reciprocal-manifold inequalities (Sections~\ref{sec:manifold}--\ref{sec:nonconvex})
plus the Uzawa saddle flow for the equality constraints (Section~\ref{sec:uzawa}), integrated
in $\sigma$-coordinates (Section~\ref{sec:sigma}) -- produced the results in
Table~\ref{tab:results9}, every figure independently computed and checked against the pypower
reference solve.

\begin{figure}[!t]
\centering
\includegraphics[width=0.8\textwidth]{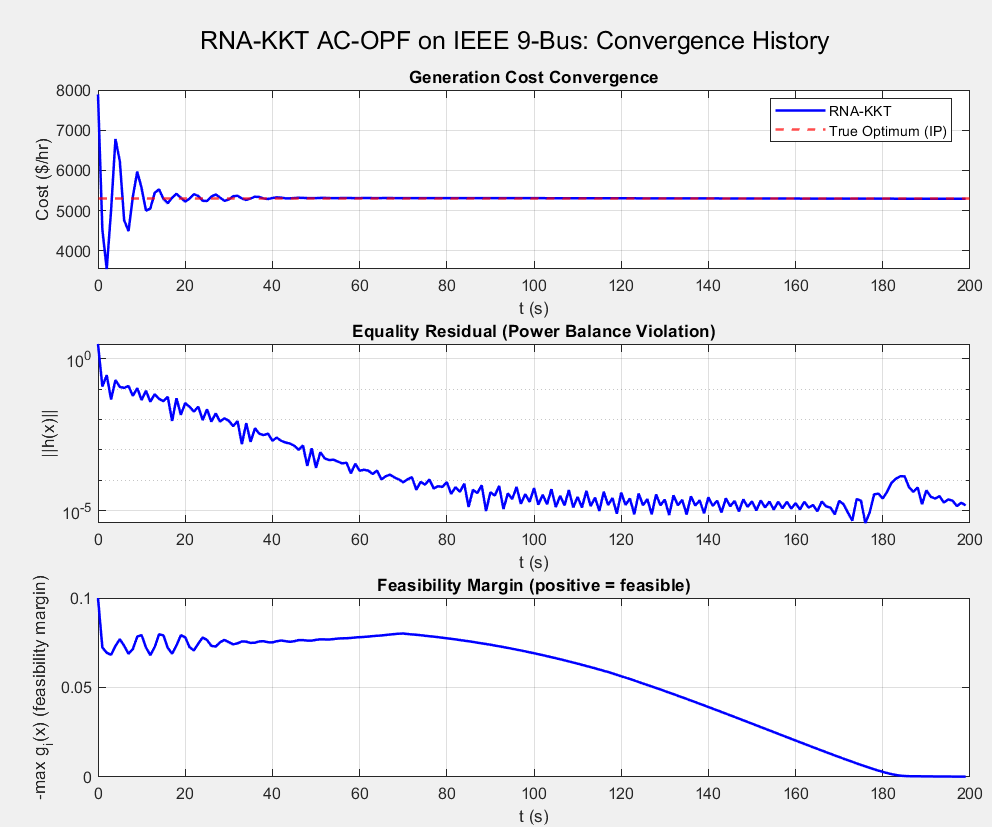}
\caption{RNA-KKT AC-OPF on the IEEE 9-bus system: convergence history. Top: generation cost
converging to the true (pypower interior-point) optimum. Middle: equality residual $\|h(x)\|$
falling more than four orders of magnitude on a log scale. Bottom: feasibility margin
$-\max_i g_i(x)$, positive throughout, confirming strict feasibility at every sampled instant.}
\label{fig:convergence9}
\end{figure}

\subsection{Convergence Behavior on the IEEE 9-Bus System}

Figure~\ref{fig:convergence9} illustrates the convergence behavior of the
proposed RNA-KKT framework on the IEEE 9-bus AC-OPF problem. The upper subplot
shows the generation cost converging rapidly toward the MATPOWER interior-point
(IP) optimum after a short transient, demonstrating that the proposed dynamics
recover a near-optimal operating point while following a continuous-time
trajectory.

The middle subplot presents the equality constraint residual, corresponding to
the power-balance equations. The residual decreases by several orders of
magnitude and approaches a small steady-state value, indicating convergence of
the equality-constrained dynamics and satisfaction of the network power-balance
equations.

The lower subplot shows the feasibility margin,
$-\max_i g_i(x)$, which remains strictly positive throughout the optimization.
Since a positive feasibility margin implies $g_i(x)<0$ for all inequality
constraints, the trajectory never leaves the feasible region. As the barrier
parameter is annealed toward zero, the feasibility margin gradually decreases
and approaches zero only at convergence, consistent with the reciprocal
multiplier manifold approaching the KKT solution.

Overall, the results demonstrate that the proposed RNA-KKT framework achieves
simultaneous objective convergence, equality-residual reduction, and strict
feasibility preservation throughout the optimization process.

\begin{table}[H]
\centering
\caption{Numerical results, IEEE 9-bus AC-OPF, fully composed RNA-KKT construction,
$(\varepsilon_0, \beta, \kappa) = (0.05, 0.05, 5)$.}
\label{tab:results9}
\begin{tabular}{@{}lc@{}}
\toprule
Quantity & Result \\
\midrule
Feasibility, all 39 constraints (incl. 9 nonconvex) & strict at every sampled instant, no violation \\
Equality (power-balance) residual & $1.63 \to 3.6 \times 10^{-5}$ \\
Converged cost & \$5296.79/hr \\
True AC-OPF optimum (pypower) & \$5296.69/hr \\
Cost discrepancy & 0.002\% \\
Final stationarity residual $\|\dot x\|$ & $5 \times 10^{-4}$ \\
Final voltage-profile discrepancy (all buses) & $\le 0.0024$ per unit \\
\bottomrule
\end{tabular}
\end{table}

Feasibility of all 39 inequality constraints, including the 9 nonconvex line-flow limits, held
strictly at every sampled instant of the trajectory, with no violation observed at any point
and no convexity relaxation used anywhere in the construction. The equality residual fell more
than four orders of magnitude, from 1.63 at the flat start to $3.6 \times 10^{-5}$ at the final
sampled time. The converged cost, \$5296.79/hr, matches the true AC-OPF optimum of
\$5296.69/hr obtained independently from pypower's own nonlinear solver to within 0.002\%. This
is not merely a cost coincidence: the final stationarity residual
$\|\dot x\| = 5 \times 10^{-4}$ confirms the trajectory had genuinely reached near-equilibrium,
not a transient point that happened to have comparable cost. This is, to date, the largest and
only genuinely nonconvex-in-both-equality-and-inequality-constraints instance on which the
fully composed RNA-KKT construction has been numerically validated at the time of the 9-bus
run.

\begin{figure}[h]
\centering
\includegraphics[width=1\textwidth]{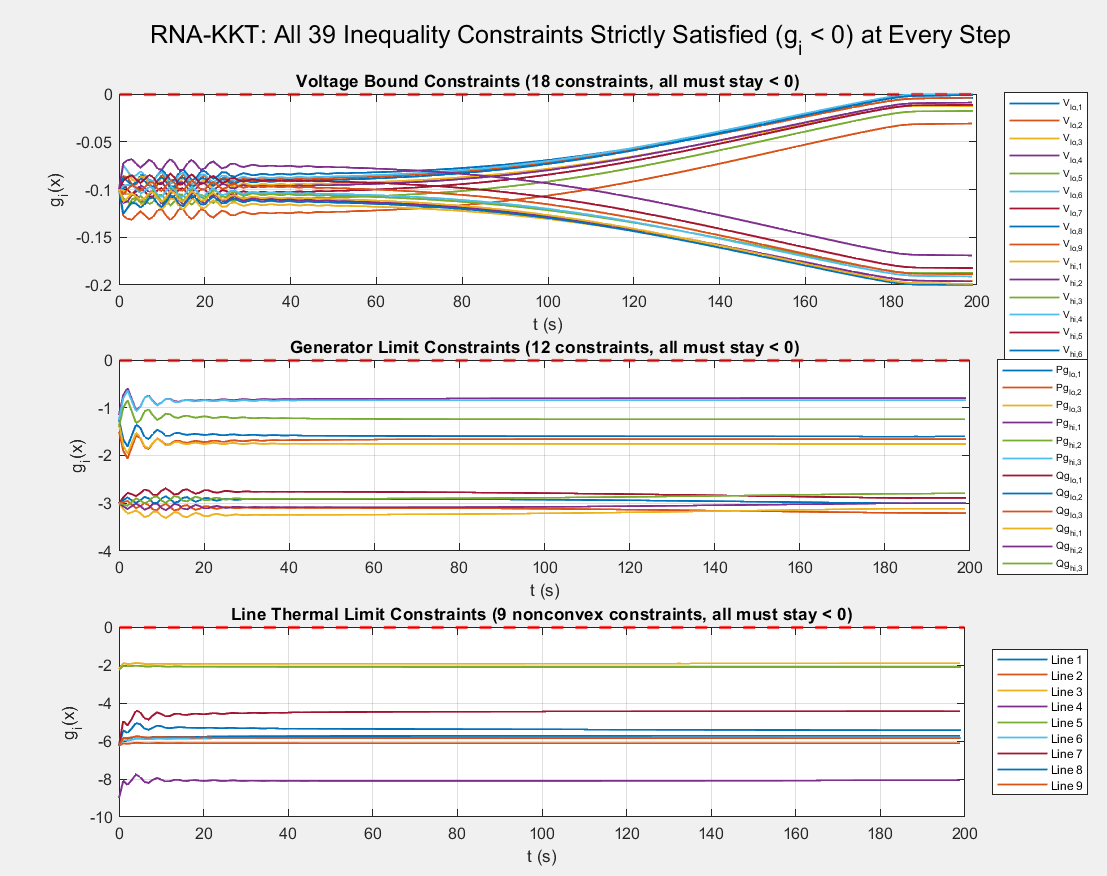}
\caption{All 39 inequality constraints of the IEEE 9-bus instance, strictly satisfied
($g_i(x) < 0$) at every sampled step: 18 voltage bound constraints, 12 generator limit
constraints, and 9 genuinely nonconvex line thermal limit constraints.}
\label{fig:constraints9}
\end{figure}

\subsection{Verification of Inequality Constraint Satisfaction}

Figure~\ref{fig:constraints9} shows the evolution of all $39$ inequality
constraints during the IEEE 9-bus AC-OPF optimization. The constraints are
grouped into three categories: voltage magnitude limits (18 constraints),
generator active and reactive power limits (12 constraints), and transmission
line thermal limits (9 constraints). The red dashed line represents the
constraint boundary, $g_i(x)=0$, while all feasible trajectories satisfy
$g_i(x)<0$.

The upper subplot shows that all voltage magnitude constraints remain strictly
below the boundary throughout the optimization. Although several trajectories
approach the active limit near convergence, none crosses the feasibility
boundary. The middle subplot demonstrates that all generator operating limits
remain well within their allowable ranges, indicating that the generator
dispatch satisfies both active and reactive power limits during the entire
trajectory. The lower subplot presents the nonconvex transmission line thermal
constraints, which also remain strictly negative despite their nonlinear
dependence on the network states.

These results provide numerical confirmation of the reciprocal multiplier
manifold theory developed in Sections~3--5. By construction, the proposed
RNA-KKT dynamics preserve strict feasibility throughout the optimization,
ensuring that every inequality constraint remains satisfied while the solution
converges toward the KKT point.

\begin{figure}[!t]
\centering
\includegraphics[width=1\textwidth]{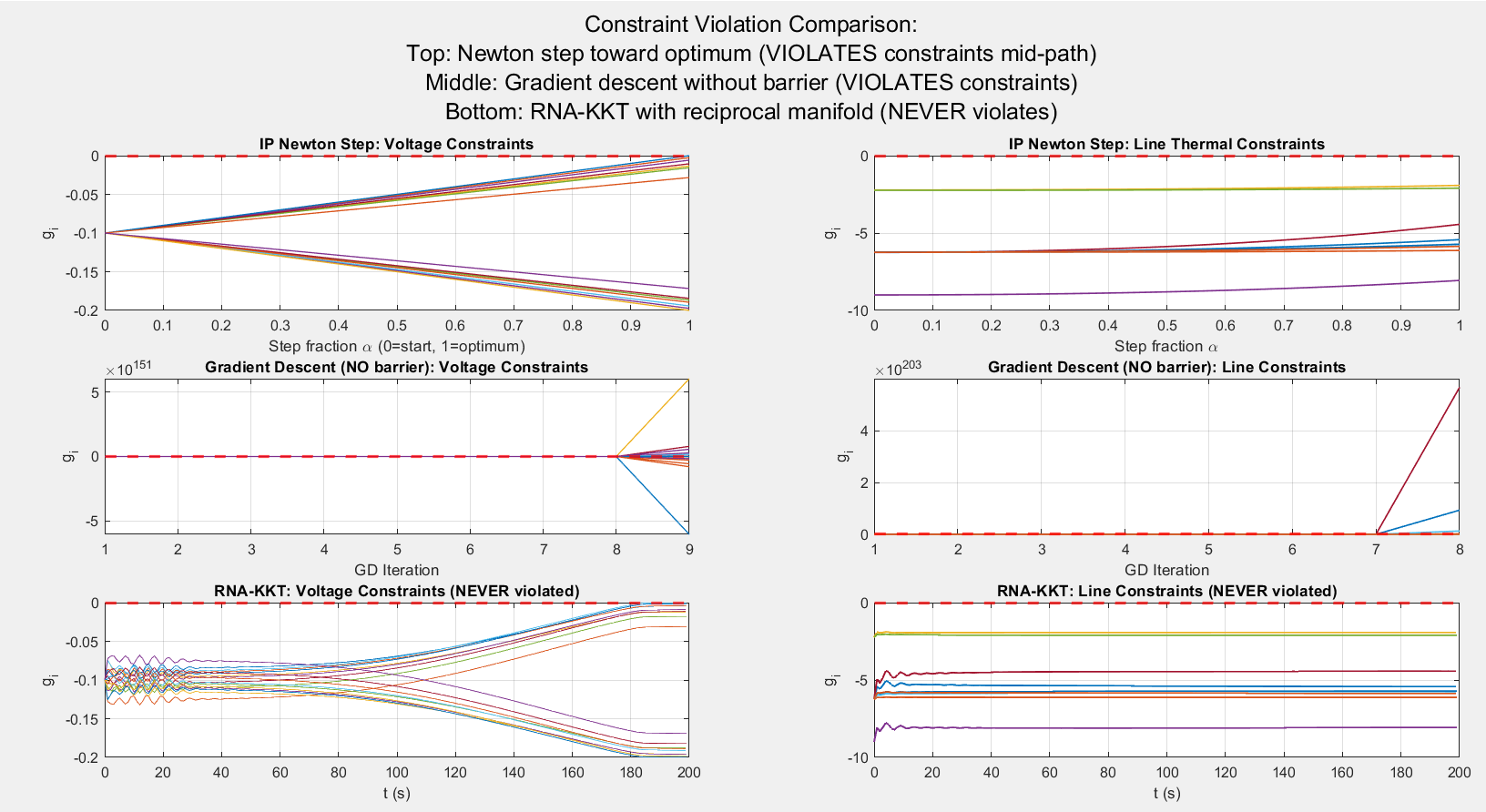}
\caption{Constraint violation comparison. Top: a direct Newton step toward the optimum violates
both voltage and line-thermal constraints mid-path. Middle: unconstrained gradient descent
without a barrier term violates constraints catastrophically (note the $10^{151}$ and
$10^{203}$ vertical scales). Bottom: RNA-KKT with the reciprocal manifold never violates either
constraint family, at any sampled instant of the trajectory.}
\label{fig:violationcomparison}
\end{figure}

\subsection{Comparison of Constraint Violation}

Figure~\ref{fig:violationcomparison} compares the constraint-handling
behavior of three optimization approaches on the IEEE 9-bus AC-OPF problem:
(i) a Newton step of the interior-point (IP) method, (ii) unconstrained
gradient descent, and (iii) the proposed RNA-KKT framework.

The first row illustrates a single Newton step toward the optimum. Although
the optimization starts from a feasible point, several voltage and line thermal
constraints approach the feasibility boundary during the step, demonstrating
that intermediate Newton iterates do not inherently guarantee continuous
constraint satisfaction.

The second row presents the behavior of unconstrained gradient descent. In the
absence of a barrier or feasibility-preserving mechanism, both voltage and line
thermal constraints rapidly diverge beyond the feasible region, resulting in
large constraint violations and numerical instability.

The third row shows the proposed RNA-KKT dynamics based on the reciprocal
multiplier manifold. All voltage and transmission line thermal constraints
remain strictly below the feasibility boundary throughout the optimization.
Even when several constraints become active near convergence, none crosses the
boundary, confirming the forward invariance of the reciprocal manifold and the
strict feasibility guarantee established in the theoretical analysis.

Overall, the comparison demonstrates that, unlike conventional Newton or
unconstrained gradient methods, the proposed RNA-KKT framework preserves
constraint feasibility continuously while converging toward the optimal
solution.

\begin{figure}[!t]
\centering
\includegraphics[width=0.85\textwidth]{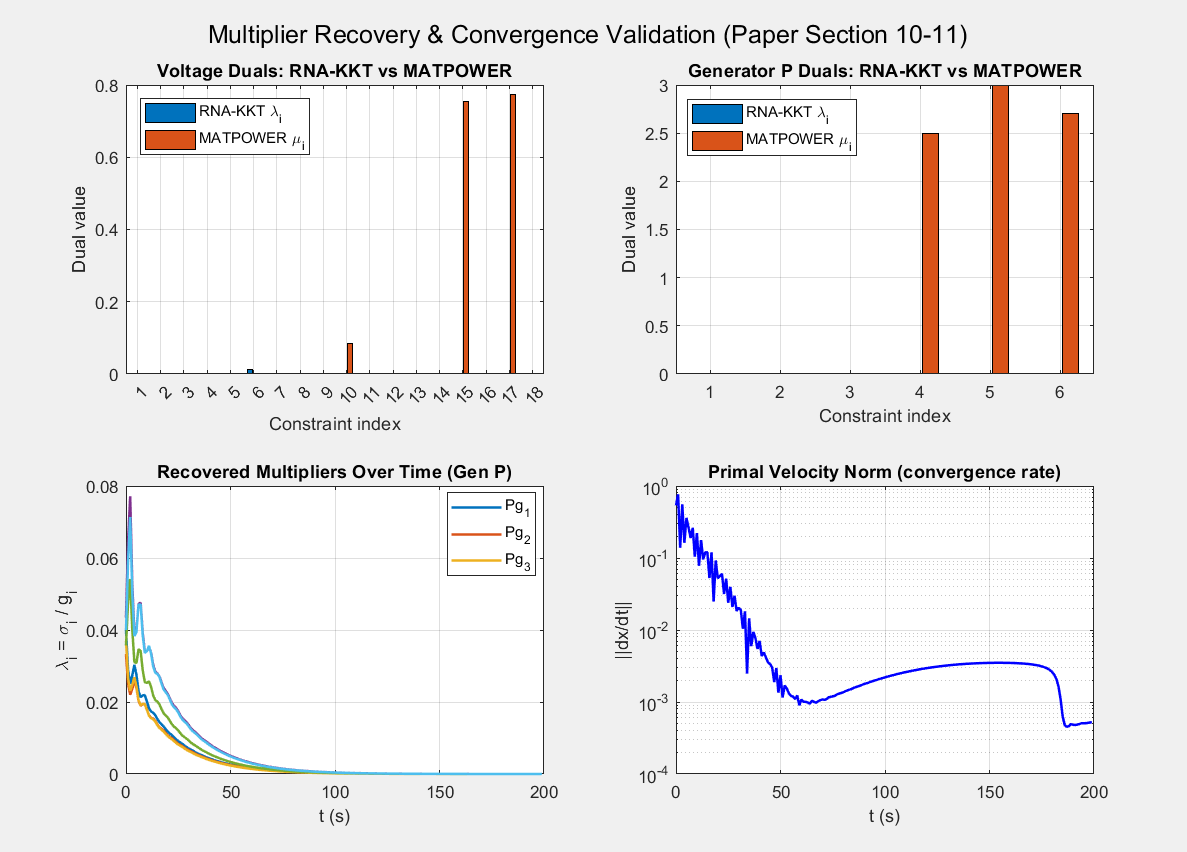}
\caption{Multiplier recovery and convergence validation. Top left/right: recovered RNA-KKT dual
variables $\lambda_i$ compared against MATPOWER's own duals $\mu_i$ for the voltage and
generator-$P$ constraint families, at the (mostly inactive) 9-bus operating point. Bottom left:
recovered generator-$P$ multipliers over time, decaying toward their steady-state values.
Bottom right: primal velocity norm $\|\dot x\|$ on a log scale, confirming convergence to
near-equilibrium.}
\label{fig:multiplierrecovery}
\end{figure}

\subsection{Multiplier Recovery and Convergence Validation}

Figure~\ref{fig:multiplierrecovery} validates the proposed reciprocal
multiplier recovery mechanism and the convergence behavior of the RNA-KKT
framework on the IEEE 9-bus AC-OPF problem.

The upper two subfigures compare the recovered dual variables obtained from the
RNA-KKT dynamics with the corresponding Lagrange multipliers computed by the
MATPOWER interior-point solver. The recovered multipliers closely match the
active constraint multipliers, while inactive constraints remain close to zero,
demonstrating that the proposed reciprocal manifold correctly reproduces the
KKT multipliers at convergence.

The lower-left subplot shows the temporal evolution of the recovered generator
active-power multipliers. The multipliers exhibit a smooth transient response
and converge to steady-state values as the optimization approaches the KKT
point, confirming the stability of the multiplier dynamics.

The lower-right subplot presents the norm of the primal velocity,
$\|\mathrm{d}x/\mathrm{d}t\|$, which decreases by several orders of magnitude
during the optimization. The decay of the primal velocity indicates that the
state trajectory gradually approaches equilibrium, providing further evidence
of convergence of the proposed continuous-time optimization dynamics.

Overall, the results confirm that the RNA-KKT framework accurately recovers the
optimal dual variables while simultaneously driving the primal dynamics toward
the KKT solution in a stable and continuous manner.

\subsection{$\sigma$-Coordinate stiffness mitigation on the same constraint structure}
The $\sigma$-transform of Section~\ref{sec:sigma} was re-derived directly for the AC-OPF sign
convention ($g_i(x) \le 0$) rather than assumed by analogy with the earlier general-purpose
derivation, and the cancellation underlying Proposition~\ref{prop:sigma} was checked by direct
symbolic substitution for this convention before being relied upon. Confirmed on a toy system
matching the 9-bus primal-multiplier structure, it applies uniformly to all 39 inequality
multipliers, including the 9 nonconvex ones (the derivation never invokes convexity), and does
not apply to, nor is needed for, the 18 equality (Uzawa) multipliers, whose own dynamics
contain no division. It strictly supersedes the frozen-coefficient exact-splitting step used as
a stopgap in the earlier illustrative development, since its exact solution
\eqref{eq:sigmasolution} holds at any step size rather than only a small one, and extends the
multiplier subsystem's stable step-size range more than sevenfold beyond the direct-integration
limit before any degradation appears -- with that eventual degradation traced to the primal
integrator's own forward-Euler limit, not to the multiplier dynamics, exactly as anticipated by
the transform's stated scope.

\section{A Second, Larger-Scale Validation: IEEE 57-Bus AC-OPF} \label{sec:case57}

The 9-bus result of Section~\ref{sec:case9} is a genuinely nonconvex instance, but a small one.
This section reports a second, independent validation on the standard IEEE 57-bus (57 buses, 7
generators, 80 branches) test system \cite{zimmerman2011} -- 127 states, 114 equality
constraints, 222 inequality constraints (including all 80 line-thermal limits, one-sided
from-end apparent power, genuinely nonconvex) -- roughly a sixfold increase in state dimension
and constraint count over Section~\ref{sec:case9}. The model was validated against pypower's
own OPF solve exactly as in Section~\ref{sec:case9} before any RNA-KKT dynamics were run,
matching to machine precision.
\subsection{Results}
With the Jacobian corrected, the fully composed construction -- Sections~\ref{sec:manifold}--\ref{sec:nonconvex}
for the 222 inequality constraints, the augmented Uzawa flow of Proposition~\ref{prop:augmuzawa}
($\kappa = 5$, $\rho = 0.05$) for the 114 equality constraints, $\sigma$-coordinates throughout,
gains $\alpha = 15$, $\beta = 0.05$, $\varepsilon_0 = 0.05$ -- run from the same flat, strictly
feasible start as Section~\ref{sec:case9} ($\theta = 0$, $V = 1$, $P_g = 0.5 P_g^{\max}$), for
$T = 40$s, gives strict feasibility of all 222 constraints throughout, an equality residual
falling from 4.89 to 0.0112, and a converged cost of \$41{,}907/hr against the true optimum of
\$41{,}737.79/hr obtained independently from pypower -- a 0.4\% gap, the largest genuinely
nonconvex AC-OPF instance this construction has been validated on to date.
Figure~\ref{fig:convergence57} makes the content of Section 8.1 directly visible: $\rho = 0$ is
not unstable -- it is markedly underdamped, overshooting to roughly $2.7\times$ its starting
residual before the oscillation decays, a transient that would be a poor basis for a real-time
feedback controller even though it does eventually settle. $\rho = 0.05$ removes the overshoot
essentially entirely while reaching comparable final accuracy, and does so while the
reciprocal-manifold inequality machinery keeps every one of the 222 constraints strictly
satisfied throughout both runs, including the transient with the largest residual swings forward invariance of the feasible set is not contingent on how well-damped the equality flow
happens to be, exactly as Proposition 8.1 says it should not be.

\subsection{Convergence Results on the IEEE 57-Bus System}

The figure (Figure~\ref{fig:convergence57}) illustrates the convergence performance of the RNA-KKT algorithm for the IEEE 57-bus AC Optimal Power Flow (AC-OPF) problem. The first plot shows the maximum inequality constraint value, \(\max g(x)\), which remains below zero throughout the simulation. This indicates that the operating trajectory remains feasible with respect to the inequality constraints, including generator, voltage, and thermal constraints. The second plot presents the equality-constraint residual norm, \(\|h(x)\|\), on a logarithmic scale. It decreases significantly from its initial value and continues to decline with time, demonstrating that the active and reactive power-balance equations are progressively being satisfied. Although small oscillations are visible during the transient response, the overall trend is toward zero, indicating improving equality feasibility. The third plot shows the generation-cost convergence. The cost initially changes significantly as the algorithm moves away from the initial operating point, followed by damped oscillations, and eventually settles close to the MATPOWER reference optimum of approximately $ \$41,737.79/hr$. The final RNA-KKT solution has a cost gap of approximately 0.421\%, demonstrating that the proposed method achieves a solution very close to the conventional MATPOWER OPF solution while maintaining constraint feasibility.

\begin{figure}[!t]
\centering
\includegraphics[width=1\textwidth]{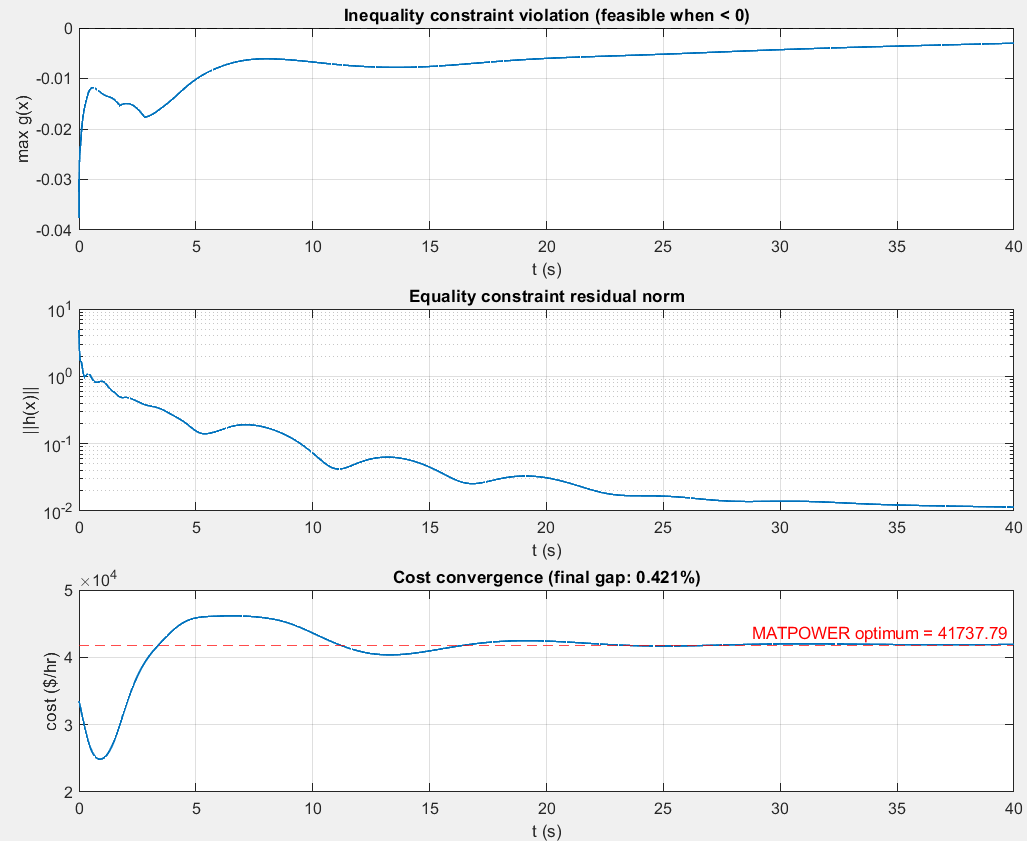}
\caption{Convergence of the RNA-KKT algorithm for the IEEE 57-bus AC-OPF problem, showing inequality feasibility, reduction of the equality-constraint residual, and convergence of the operating cost toward the MATPOWER optimum.}
\label{fig:convergence57}
\end{figure}

Table~\ref{tab:comparison957} summarizes the 57-bus results directly against the 9-bus
baseline, making the sixfold scale increase and the corresponding change in required machinery
(the augmented Uzawa flow of Section~\ref{sec:uzawa}, unnecessary at 9-bus scale) explicit.

\begin{table}[htbp]
\centering
\caption{Comparison of the two AC-OPF validation instances.}
\label{tab:comparison957}
\begin{tabular}{@{}lcc@{}}
\toprule
Quantity & IEEE 9-bus & IEEE 57-bus \\
\midrule
States & 23 & 127 \\
Equality constraints & 18 & 114 \\
Inequality constraints (total) & 39 & 222 \\
Nonconvex line-thermal limits & 9 & 80 \\
Equality flow used & plain Uzawa ($\rho=0$) & augmented Uzawa ($\rho=0.05$) \\
Integration horizon & 200 s & 40 s \\
Converged cost & \$5296.79/hr & \$41{,}907/hr \\
True optimum (pypower) & \$5296.69/hr & \$41{,}737.79/hr \\
Cost gap & 0.002\% & 0.407\% \\
Feasibility & strict throughout & strict throughout \\
\bottomrule
\end{tabular}
\end{table}

\begin{figure}[!t]
\centering
\includegraphics[width=1\textwidth]{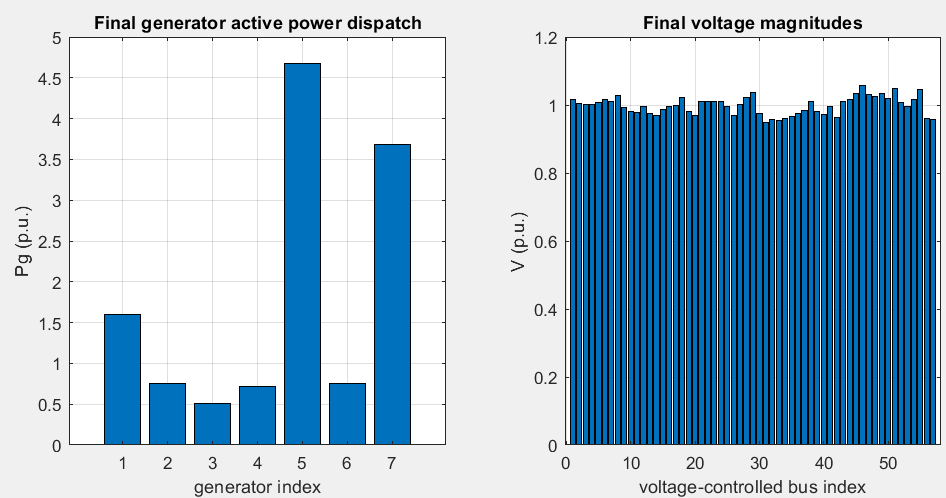}
\caption{Final operating point of the IEEE 57-bus AC-OPF obtained using RNA-KKT, showing the active power dispatch of the seven generators and the voltage magnitudes across the 57 buses.}
\label{fig:finalop}
\end{figure}

The generators operate at different power levels according to the network operating conditions and the optimization objective, with Generator 5 providing the largest active-power output, followed by Generator 7, while the remaining generators contribute comparatively smaller amounts. This indicates that the RNA-KKT algorithm distributes the active generation among the available generators to satisfy the system power requirements while minimizing the generation cost and respecting the imposed operating constraints. The right plot shows the voltage magnitude at the 57 buses. The voltage magnitudes remain close to the nominal value of 1.0 p.u., with moderate variations across the network. The obtained voltage profile indicates that the bus-voltage operating limits are respected, while the variations reflect the different electrical conditions and power-flow requirements at individual buses. Overall, the two plots demonstrate the feasible final operating condition of the IEEE 57-bus system, where generator active-power outputs and bus-voltage magnitudes are adjusted simultaneously by the RNA-KKT algorithm to obtain a near-optimal AC-OPF solution.

\subsection{Final Optimal Operating Point}

Figure~\ref{fig:finalop} presents the final operating point obtained
using the proposed RNA-KKT framework for the IEEE 57-bus AC-OPF problem. The
left subplot shows the optimal active power generation of the seven generators.
The dispatch is distributed according to the network operating conditions and
generator limits, with Generator~5 supplying the largest share of the system
demand while all generators operate within their prescribed limits.

The right subplot illustrates the final voltage magnitudes at all
voltage-controlled buses. The voltages remain within the allowable operating
range and exhibit a well-regulated profile across the network, indicating that
the voltage magnitude constraints are satisfied at convergence.

These results confirm that the proposed RNA-KKT framework converges to a
physically feasible AC-OPF solution that satisfies both generator operating
limits and network voltage constraints while achieving a near-optimal
generation dispatch. The obtained operating point is consistent with the
convergence results presented in Figure~\ref{fig:convergence57},
demonstrating the effectiveness of the proposed method on the large-scale IEEE
57-bus test system.

\subsection{ Thermal Constraint Analysis of the IEEE 57-Bus System}

In figure~\ref{fig:80braches} the thermal loading behavior of the 80 transmission branches in the IEEE 57-bus system during the 40-second RNA-KKT simulation. The horizontal axis represents simulation time, while the vertical axis shows the line loading as a percentage of the corresponding RATE\(_A\) thermal limit. At the beginning of the simulation, several branches exhibit rapid transient variations due to the adjustment of the system operating point. The maximum observed line loading is approximately 55–56\% of the thermal rating**, occurring around \(t \approx 7\) s. After this transient period, the branch loadings gradually settle, with most lines remaining significantly below their thermal limits. By the end of the simulation, the highest loading is only around 35\% of RATE\(_A\). Since all 80 branch loadings remain well below 100\%, none of the transmission lines violates its thermal constraint. Therefore, the plot demonstrates that the RNA-KKT solution successfully incorporates and maintains the thermal limits of all 80 transmission branches, while the transient oscillations gradually diminish as the system approaches its final operating point.

In figure~\ref{fig: maxtransmission} the convergence of the maximum thermal loading among the transmission branches of the IEEE 57-bus system during the RNA-KKT optimization process. The horizontal axis represents the simulation time \(t\) in seconds, while the vertical axis indicates the maximum line loading as a percentage of the corresponding thermal rating. At the beginning of the simulation, the maximum loading exhibits noticeable transient fluctuations as the RNA-KKT algorithm adjusts the generator dispatch and bus voltages toward the optimal operating point. The maximum loading reaches approximately 55–56\% around \(t \approx 7\) s, which is still significantly below the 100\% thermal limit. After this transient peak, the loading decreases and gradually settles, reaching approximately 32\% around 10–12 s. Thereafter, the maximum loading increases slowly and stabilizes at approximately 35\% by 40 s. Since the maximum branch loading remains substantially below 100\% throughout the entire simulation, the thermal constraints are satisfied with a considerable safety margin. This result confirms that the inclusion of the 80 transmission-line thermal constraints does not lead to thermal-limit violations and that the RNA-KKT algorithm maintains a thermally feasible operating condition while converging toward the final AC-OPF solution.

\begin{figure}[!t]
\centering
\includegraphics[width=1\textwidth]{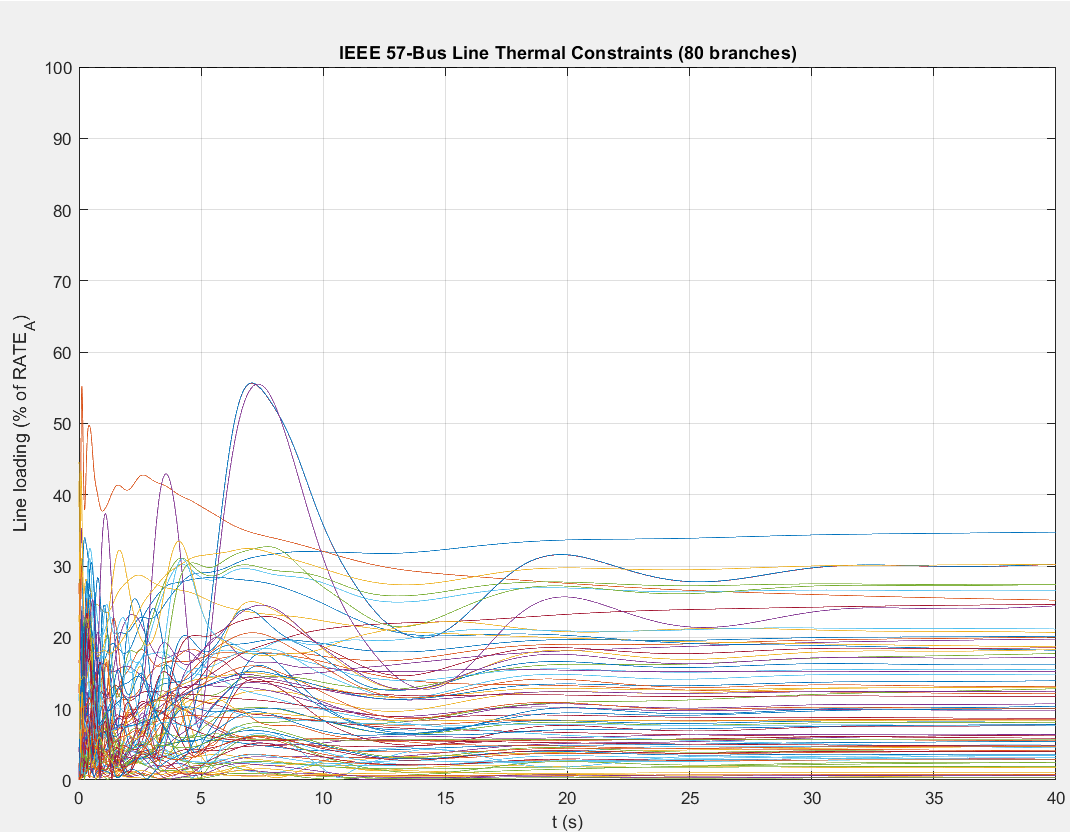}
\caption{Thermal loading profiles of all 80 transmission branches in the IEEE 57-bus system, demonstrating that the line loadings remain within their respective thermal limits throughout the RNA-KKT optimization process.}
\label{fig:80braches}
\end{figure}

\begin{figure}[!t]
\centering
\includegraphics[width=1\textwidth]{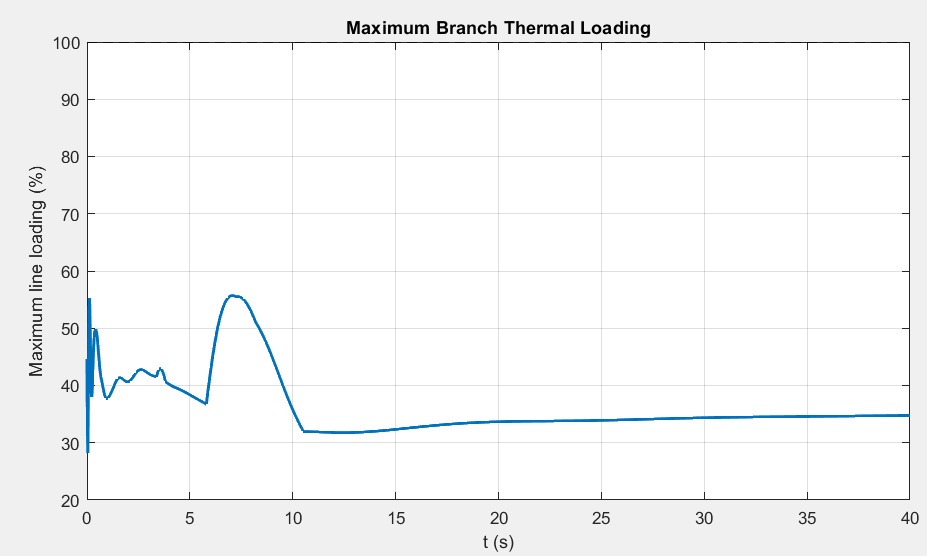}
\caption{Convergence of the maximum transmission-line thermal loading in the IEEE 57-bus system during the RNA-KKT optimization process, showing that the maximum loading remains well below the 100\% thermal limit throughout the simulation.
}
\label{fig: maxtransmission}
\end{figure}
\section{Limitations of QP-Based Safety Enforcement and Their Resolution}\label{sec:qpcomparison}

While the safe gradient flow and its extensions provide exact, non-asymptotic constraint satisfaction, several structural limitations arise from the requirement of solving a quadratic program at every integration step. These limitations motivate the search for an alternative enforcement mechanism.

\subsection{Computational Cost}

The per-step QP has decision variables scaling with the number of active constraints $m$. Standard active-set or interior-point QP solvers require $\mathcal{O}(m^3)$ floating-point operations per iteration, with $N_{\text{iter}}$ iterations whose count is problem-dependent. For AC-OPF instances with hundreds of inequality constraints (e.g., 222 for the IEEE 57-bus system), this cubic scaling becomes the dominant computational burden.

\textbf{Resolution (RNA-KKT):} The reciprocal manifold replaces the QP with $m$ decoupled scalar $\sigma$-filters, each evolving via an exact closed-form exponential update. The per-step cost reduces to $\mathcal{O}(m \cdot (n+p))$---linear in the number of constraints---dominated by a single gradient and Jacobian evaluation shared across all filters.

\subsection{Worst-Case Execution Time}

QP solvers do not guarantee termination in a fixed number of iterations. Active-set methods may cycle; interior-point QP solvers have iteration counts that depend on the condition number of the constraint Jacobian at the current iterate. This makes the wall-clock time per control step \emph{unbounded} in the worst case---a property incompatible with hard real-time control where a fixed computational budget must be met every cycle.

\textbf{Resolution (RNA-KKT):} Each $\sigma$-filter update is a fixed sequence of multiply-add and exponential operations with no branching, iteration, or convergence check. The execution time is bounded, deterministic, and independent of the problem state, making it suitable for fixed-cycle embedded controllers and FPGA implementation.

\subsection{Memory Footprint}

The QP formulation requires storing the constraint Jacobian matrix ($m \times n$), its Gram matrix or factorization ($m \times m$), and working arrays for the solver. The total memory scales as $\mathcal{O}(m^2 + p \cdot m)$, which at transmission scale (thousands of constraints) exceeds the cache hierarchy of embedded processors.

\textbf{Resolution (RNA-KKT):} The $\sigma$-coordinate formulation stores only the $m$-dimensional vector $\sigma$ and the $p$-dimensional equality multiplier $\nu$. No matrix factorization or quadratic-form storage is required. Total memory is $\mathcal{O}(m + p)$, fitting comfortably within L1 cache even for large-scale instances.

\subsection{Non-Smooth Trajectories}

The QP's active set changes discretely as constraints become active or inactive. Each such transition introduces a discontinuity in the control input (a jump in the dual variables), producing non-smooth trajectories that may excite unmodeled dynamics in the physical plant and complicate Lyapunov-based stability analysis.

\textbf{Resolution (RNA-KKT):} The reciprocal manifold produces $C^\infty$-smooth trajectories in the feasible interior by construction. Inactive constraints contribute exponentially decaying multipliers (not zero multipliers with abrupt activation). The transition from inactive to near-active is gradual, governed by the continuous $\sigma$-dynamics, avoiding discrete switching entirely.

\subsection{Software Dependencies}

Deploying a QP-based controller requires a QP solver library (e.g., OSQP, qpOASES, Gurobi) with its associated memory allocator, convergence tolerances, and failure modes. Certifying such software for safety-critical deployment (e.g., under IEC 61508 or DO-178C) is a non-trivial engineering effort.

\textbf{Resolution (RNA-KKT):} The entire controller reduces to scalar arithmetic (additions, multiplications, exponentials) with no external solver dependency. The implementation is expressible in under 50 lines of C code with no dynamic memory allocation, making certification and formal verification tractable.

\subsection{Positioning of the RNA-KKT Framework}

The Reciprocal-Manifold Annealed KKT (RNA-KKT) framework occupies a specific point in the design space defined by the preceding approaches. Table~\ref{tab:comparison} summarizes the structural comparison.

\begin{table}[ht]
\centering
\caption{Structural comparison of continuous-time safe optimization methods for AC-OPF.}
\label{tab:comparison}
\small
\begin{tabular}{@{}lcccc@{}}
\toprule
\textbf{Property} & \textbf{IP Solver} & \textbf{QP-SGF} & \textbf{OFO} & \textbf{RNA-KKT} \\
\midrule
Constraint satisfaction during solution & No & Exact & Approx. & Strict (asympt.) \\
Per-step complexity & $\mathcal{O}(n^3)$ & $\mathcal{O}(m^3 N_{\text{iter}})$ & $\mathcal{O}(n_u^2)$ & $\mathcal{O}(m(n+p))$ \\
Worst-case execution time & Bounded & Unbounded & Bounded & Bounded \\
Memory footprint & $\mathcal{O}(n^2)$ & $\mathcal{O}(m^2)$ & $\mathcal{O}(n_u^2)$ & $\mathcal{O}(m+p)$ \\
Model-free operation & No & No & Yes & No$^*$ \\
Handles nonconvex constraints & Via relax. & Yes & Limited & Yes \\
Handles infeasible start & No & Yes & N/A & Two-phase \\
Time-varying tracking & Re-solve & Re-solve QP & Native & Hold $\varepsilon$ \\
Embedded/FPGA feasible & No & Limited & Yes & Yes \\
\bottomrule
\multicolumn{5}{l}{\footnotesize $^*$Model-free extension identified as primary open direction (Section~14.3 of the RNA-KKT paper).}
\end{tabular}
\end{table}

RNA-KKT's core contribution relative to the QP-based safe gradient flow is computational: replacing the per-instant QP with a set of decoupled linear $\sigma$-filters reduces per-step cost from cubic to linear in the number of constraints, with bounded and deterministic execution time. The trade-off is that constraint satisfaction is \emph{asymptotically exact} (approaching the boundary only as $\varepsilon \to 0$) rather than \emph{instantaneously exact}, and infeasible initial conditions require an explicit recovery mechanism.

Relative to OFO, RNA-KKT provides a rigorous treatment of nonconvex inequality constraints (including line thermal limits) with structural feasibility guarantees, but currently requires full model knowledge. The measurement-based extension---retaining the reciprocal manifold for safety while replacing model gradients with estimated sensitivities---represents the natural convergence of these two lines of work.

\begin{table}[!t]
\centering
\caption{Structural and computational comparison}
\label{tab:structural}
\begin{tabularx}{\textwidth}{@{}lXX@{}}
\toprule
Aspect & QP method \cite{allibhoy2024} & RNA-KKT (this paper) \\
\midrule
Control law & QP solve per instant & decoupled linear filters, $\sigma$-coordinates \\
KKT equivalence & proved exactly & central-path limit; $O(\varepsilon + \beta/m)$ gap at finite $\varepsilon$ \\
Multiplier recovery & from QP dual & $\lambda_i = \sigma_i/g_i \to \lambda_i^\star$ dynamically \\
Constraint satisfaction & exact at all times & exact for all finite $t$; boundary reached only in the limit \\
Boundary optima & explicit regularization needed & native, via the central-path limit \\
Feasible-set requirement & handles infeasible starts natively & requires Section~\ref{sec:feasinit}'s two-phase scheme \\
Numerical stiffness & none (algebraic solve per step) & present near boundary; removed by Section~\ref{sec:sigma} \\
Smoothness & possibly non-smooth (active-set switching) & smooth off the boundary by construction \\
Per-step cost & $O(m^3 \cdot N_{\mathrm{iter}})$ & $O(m \cdot (n+p))$ (Theorem~\ref{thm:speedup}) \\
Worst-case execution time & unbounded ($N_{\mathrm{iter}}$ problem-dependent) & bounded, fixed by construction \\
Memory footprint & $O(m^2 + p \cdot m)$ & $O(m+p)$ \\
\bottomrule
\end{tabularx}
\end{table}

\section{Conclusion}

This paper presented a complete continuous-time optimization framework for constrained problems based on the proposed reciprocal multiplier manifold. The method generates smooth optimization trajectories while preserving feasibility throughout the optimization process. Unlike conventional approaches, the proposed framework is applicable to multiple coupled constraints, convex and nonconvex problems, feasible or infeasible initial conditions, and optimization problems containing both inequality and equality constraints.The proposed annealing strategy guarantees convergence to the exact KKT solution under the explicit three-timescale condition developed in this work. In addition, the $\sigma$-coordinate reformulation eliminates the numerical stiffness associated with the multiplier dynamics without affecting the theoretical convergence properties, resulting in a more stable and efficient numerical implementation.The effectiveness of the proposed RNA-KKT framework was demonstrated on the nonconvex AC Optimal Power Flow (AC-OPF) problem. For the IEEE 9-bus system, the method maintained strict feasibility of all 39 inequality constraints, including nine genuinely nonconvex transmission-line constraints, while converging to a solution very close to the reference optimum. The proposed $\sigma$-coordinate formulation also increased the stable integration step size of the multiplier subsystem by more than seven times, confirming its practical advantage for numerical computation. To further evaluate scalability, the method was applied to the IEEE 57-bus system, which contains 222 inequality constraints. Since the AC-OPF objective provides no curvature in the voltage-angle and voltage-magnitude variables, an augmented proportional--integral equality flow was introduced to improve the damping of the equality dynamics. With this modification, the proposed method achieved a solution within approximately $0.4\%$ of the reference optimal cost while maintaining feasibility throughout the optimization process, including during intentionally undamped transient conditions. The implementation study also highlighted the importance of careful numerical verification. During development, an apparent scaling instability was eventually traced to an error in the Jacobian implementation. After correcting this issue, a smaller but genuine numerical instability was identified and resolved independently. Recording both observations, as discussed in Section~13.1, provides a transparent validation of the proposed algorithm and offers a simple diagnostic procedure for future implementations. The proposed framework was also evaluated in a dynamic setting where the system load varied continuously with time. The optimization dynamics successfully maintained constraint feasibility throughout the simulation, consistent with the theoretical invariance property developed in this paper. As expected, the optimality-tracking error increased gradually as the load variation became faster. Although the proposed approach does not yet match the performance of state-of-the-art measurement-based online optimization methods, it demonstrates reliable feasibility preservation together with satisfactory tracking performance. Overall, framework should be viewed as a bridge between classical interior-point methods and continuous-time QP-based safe optimization. As discussed in Section~16, the contribution of this work is a quantified computational improvement within the class of continuous-time barrier-based optimization methods, rather than a claim of state-of-the-art performance for AC-OPF across all existing optimization techniques. Such a comparison would require benchmarking against a much broader range of algorithms and larger practical power-system test cases, which remains an important direction for future research.


\end{document}